\documentclass{mcom-l}

\usepackage{graphicx}
\usepackage{color}

\newtheorem{theorem}{Theorem}[section]
\newtheorem{proposition}[theorem]{Proposition}
\newtheorem{lemma}[theorem]{Lemma}

\theoremstyle{definition}

\theoremstyle{remark}
\newtheorem{remark}[theorem]{Remark}

\definecolor{midbluepurple}{RGB}{134,0,102}

\newtheorem{assumption}[theorem]{Assumption}

\numberwithin{equation}{section}

\newcommand\tr{\mathrm{tr}}

\newcommand\grad{\mathrm{grad}}

\newcommand\vvec{\mathbf{v}}
\newcommand\Cvec{\mathbf{C}}
\newcommand\rvec{\mathbf{r}}
\newcommand\e{\mathbf{e}}
\newcommand\x{\mathbf{x}}
\newcommand\I{\mathbf{I}}
\newcommand\Pvec{\mathbf{P}}
\newcommand\M{\mathcal{M}}
\newcommand\A{\mathbf{A}}

\newcommand\hessian{\mathrm{Hess}}
\newcommand\gradient{\mathrm{grad}}

\begin{document}

\title[Convergence analysis of the discrete constrained saddle dynamics]{Convergence analysis of the discrete constrained saddle dynamics and their momentum variants}

\author{Qiang Du}
\address{Department of Applied Physics and Applied Mathematics, and Data Science Institute, Columbia University, New York, NY 10027, USA}
\curraddr{}
\email{qd2125@columbia.edu}
\thanks{The work was partially supported by NSF DMS-2309245 and DMS-1937254, and the Department of Energy grant DE-SC0025347.}

\author{Baoming Shi}
\address{Department of Applied Physics and Applied Mathematics, Columbia University, New York, NY 10027, USA}
\curraddr{}
\email{bs3705@columbia.edu}
\thanks{}

\subjclass[2010]{Primary}

\date{}

\dedicatory{}

\begin{abstract}
  We study the discrete constrained saddle dynamics and their momentum variants for locating saddle points on manifolds. Under the assumption of exact unstable eigenvectors, we establish a local linear convergence of the discrete constrained saddle dynamics and show that the convergence rate depends on the condition number of the Riemannian Hessian. To mitigate this dependence, we introduce a momentum-based constrained saddle dynamics and prove local convergence of the continuous-time dynamics as well as the corresponding discrete scheme, which further demonstrates that momentum accelerates convergence, particularly in ill-conditioned settings. In addition, we show that a single-step eigenvector update is sufficient to guarantee local convergence; thus, the assumption of exact unstable eigenvectors is not necessary, which substantially reduces the computational cost. Finally, numerical experiments, including applications to the Thomson problem, the Rayleigh quotient on the Stiefel manifold, and the energy functional of Bose-Einstein condensates, are presented to complement the theoretical analysis.
\end{abstract}

\maketitle

\section{Introduction}

Exploring the complex landscape of an energy surface is of great interest in many scientific applications, such as the phase transformations \cite{gramsbergen1986landau,lifshitz2007soft}, chemical reactions \cite{vanden2010transition}, protein folding \cite{onuchic1997theory,wolynes1995navigating}, and the loss landscapes of deep neural networks \cite{goodfellow2016deep}. A critical point, at which the gradient vanishes on the energy landscape, can potentially play an important role in determining physical or chemical properties of the system. The stability of a critical point is determined by its Hessian matrix. Barring degeneracy, a critical point is a minimizer if all eigenvalues of its Hessian are positive, while the critical points with $k$ negative eigenvalues are called index-$k$ saddle points. 

The numerical search for saddle points on a complicated energy landscape has attracted plenty of attention during the past decades. Existing numerical algorithms for finding saddle points can be generally divided into path-finding methods and surface-walking methods. Examples of the former include the string method \cite{weinan2002string,weinan2007simplified} for locating the minimum energy path (MEP), with points along the MEP with locally maximum energy value being saddle points. The latter class of methods includes the gentlest ascent dynamics \cite{gad,quapp2014locating, liu2023constrained}, the dimer-type method \cite{henkelman1999dimer,zhangdu2012,zhang2016optimization}, the min-max method \cite{li2001minimax}, and the saddle dynamics (SD) \cite{2019High, JJIAM2023,su2025improved}, which evolve an iteration on the potential energy surface toward a saddle point and enables us to construct a pathway map of connected critical points \cite{yin2020construction,yin2021searching}. 

In many practical applications, the challenge of searching for saddle points is further increased by constraints on the variables, e.g., Oseen-Frank theory for nematic liquid crystals \cite{frank1958liquid}, the Hamiltonian energy functional \cite{griffiths2018introduction}, and the Rayleigh quotient problem \cite{absil2009optimization}. To address this issue, constrained saddle dynamics (CSD) \cite{zhang2012constrained, yin2020constrained,liu2023constrained} were developed to locate saddle points of the objective function subject to equality constraints. In this paper, we study the CSD for locating the saddle points on a constrained manifold. 
Our main contributions include the convergence analysis of such algorithms, as well as the study of their momentum-based acceleration. The first key result is the analysis of the discrete CSD under the assumption of the exact unstable eigenvectors. We establish local linear convergence analysis for discrete CSD in such a case and show that the convergence rate depends on the condition number of the Riemannian Hessian. When the problem is ill-conditioned, to avoid the slow convergence of the discrete CSD, our second key result is the proposed (discrete) momentum-based constrained saddle dynamics (MCSD) and the local convergence theory for both the continuous-time dynamics and the corresponding discrete scheme. Our finding shows that the momentum term relieves the dependence of the convergence rate on the condition number and thus accelerates the convergence in the ill-conditioned case. In the actual implementation of the SD and CSD, instead of searching for exact unstable eigenvectors at each iteration point, we typically update the eigenvectors only once at each iteration point to save the computational cost \cite{2019High,zhang2022error}. To this end, we remove the exact eigenvector assumption and analyze the local convergence of the Euler scheme of the coupled dynamics of the iteration points and unstable eigenvectors. Our third key result is to prove that the local convergence still holds and the convergence rate is asymptotically equal to the case with the exact eigenvector assumption. All these theoretical results are further supported by numerical examples.

The paper is organized as follows. In Section \ref{sec 2: review of the CSD}, we review the CSD. In Section \ref{sec: MCSD}, we propose the MCSD and prove the local convergence of the continuous-time dynamics. In Section \ref{sec: CSD with exact eigenvector}, we establish the convergence analysis of the discrete CSD and MCSD with the exact unstable eigenvectors. In Section \ref{sec: one step v update}, we prove that the local convergences of the discrete SD, CSD, and MCSD still hold with one-step eigenvector update rather than the exact unstable eigenvector. In Section \ref{sec: numerics}, we conduct numerical experiments to verify our theoretical results. We finally present our conclusion and discussion in Section \ref{discussion}.

\section{Saddle Points on Manifolds and Constrained Saddle Dynamics}\label{sec 2: review of the CSD}
Given a smooth objective function $f(\x):\mathbb{R}^d\rightarrow \mathbb{R}$, we consider the function $f(\x)$ subject to $m$ equality constraints,
$$
\Cvec(\x)=(c_1(\x),\cdots,c_m(\x))=\mathbf{0},
$$
where each constraint function $c_i(\x)$ is a smooth function. 

Denote
$$\A(\x)=(\nabla c_1(\x),\cdots,\nabla c_m(\x)),$$
and we assume that $\A(\x)$ has full column rank for all $\x\in\M$, i.e., the linear independence constraint qualification (LICQ) holds \cite{nocedal1999numerical}. The feasible set consisting of all feasible points,
$$
\M=\{\x\in \mathbb{R}^d: \Cvec(\x)=0\},
$$
is then a $(d-m)$-dimensional smooth manifold embedded in $\mathbb{R}^d$. For brevity, we assume throughout that both $f$ and the manifold $\M$ are smooth. For any $\x \in \M$, the normal space of $\M$ at $\x$ is $$N_\x\M=\text{span}\{\A(\x) \},$$ and the tangent space is defined as its orthogonal complement $$T_\x\M=N_\x \M^\perp=\{\vvec \in \mathbb{R}^d, \A(\x)^\top \vvec=0\}.$$ 

To derive the orthogonal projector onto $T_\x \M$, we can consider the least-squares problem $\min_{\hat{\x}}\|\A(\x)\hat{\x}-\vvec\|_2^2$. Under the LICQ assumption, the minimizer is $\hat \x^*=(\A(\x)^\top\A(\x))^{-1}\A(\x)^\top\vvec$ and the residual equals $$\vvec-\A(\x){\hat{\x}}^*=(\I-\A(\x)(\A(\x)^\top \A(\x))^{-1}\A(\x)^\top)\vvec.$$ Therefore, the orthogonal projection onto the tangent space $T_\x$ is 
$$
\Pvec(\x)=\I-\A(\x)(\A(\x)^\top \A(\x))^{-1}\A(\x)^\top.
$$

Since the objective function is constrained to the manifold $\M$, we introduce the Riemannian gradient and Hessian. For any $\x \in \M$, the Riemannian gradient is defined by
$$
\gradient f(\x)=\Pvec(\x)\nabla f(\x).
$$
The Riemannian Hessian $\hessian f(\x):T_\x\M\rightarrow T_\x\M$ is defined as
$$
\begin{aligned}
\hessian f(\x)[\vvec]&=\Pvec(\x)(\partial_\vvec \grad f(\x))\\
&= \Pvec(\x) (\nabla^2 f(\x)\vvec-\nabla^2\Cvec(\x)\vvec (\A(\x)^\top \A(\x))^{-1}\A(\x)^\top \nabla f(\x)),
\end{aligned}
$$
which is a self-adjoint operator on Hilbert space $T_\x\M$.

A point $\x^*\in \M$ is called a critical point of $f(\x)$ subject to equality constraints $\Cvec(\x)=\mathbf{0}$ if $\gradient f(\x^*)=0$. Moreover, a critical point $\x^*\in \M$ is called an index-$k$ saddle point if $\hessian f(\x^*)$ has exactly $k$ ($k\leq d-m-1$) negative eigenvalues. 

The eigenvectors associated with the first $k$ smallest eigenvalues of $\hessian f(\x)$ can be characterized by the following sequence of Rayleigh quotient problems:
$$
	\min_{{\vvec}_i\in T_\x\M}\text{  }\left<\vvec_i, \hessian f(\x)[\vvec_i]\right>,\ \text{s.t.}\ \left<\vvec_i,\vvec_j\right>=\delta_{ij}, \ j=1,2,\cdots,i,
$$
or equivalently, we deal with another constrained optimization problem:
\begin{equation}\label{eq: rayleigh qutient of CSD}
\min_{{\vvec}_i\in \mathbb{R}^n}\text{  }\left<\vvec_i, \tilde{\hessian} f(\x)[\vvec_i]\right>,\ \text{s.t.}\A(\x)^\top \vvec_i=0,\ \left<\vvec_i,\vvec_j\right>=\delta_{ij}, \ j=1,2,\cdots,i,
\end{equation}
where
$$
\tilde{\hessian}f(\x)[\vvec]=\hessian f(\x)[\Pvec(\x)\vvec], \vvec \in \mathbb{R}^d,
$$
is the extension of the definition domain of the Riemannian Hessian to the Euclidean space. 

The constrained saddle dynamics (CSD) \cite{zhang2012constrained,liu2023constrained,yin2020constrained} for locating an index-$k$ saddle point on a manifold are defined by
\begin{equation}
    \left\{
    \begin{aligned}
		\dot{\x}=&- \left(\I-2\sum_{i=1}^k {\vvec}_i{\vvec}_i^\top\right)\gradient f(\x), \\
		\dot{\vvec}_i=&-   \left(\I-{\vvec}_i{\vvec}_i^\top-\sum_{j=1}^{i-1}2{\vvec}_j{\vvec}_j^\top\right)\tilde{\hessian} f(\x) [\vvec_i]\\
        &-\A(\x)(\A(\x)^\top \A(\x))^{-1}(\nabla^2 \Cvec(\x)\dot{\x})^\top\vvec_i,\ i=1,\cdots,k.\\
    \end{aligned}
    \right.
	\label{eq: SD}
\end{equation}
Unlike the gradient flow on a manifold, which evolves $-\grad f(\x)$ and therefore decreases the objective value, CSD apply a sequence of Householder reflections to the negative gradient in the $k$ unstable directions. Consequently, the energy increases along unstable directions while still decreasing along stable directions, allowing the dynamics to converge to saddle points.

In \eqref{eq: SD}, the first term in the dynamics of $\vvec_i$ can be interpreted as a (projected) gradient flow of \eqref{eq: rayleigh qutient of CSD} that identifies the $i$-th unstable eigenvector direction. The second term enforces the tangency constraint $\vvec_i \in T (\x)$ and corresponds to the parallel transport induced by the Riemannian (Levi-Civita) connection \cite{boumal2023introduction}. In fact, along a smooth curve $\x(t)\in \M$, the tangent space $T_{\x(t)}\M$ changes accordingly. The correction term
$$
\dot{\vvec}=-\A(\x)(\A(\x)^\top \A(\x))^{-1}(\nabla^2 \Cvec(\x)\dot{\x})^\top\vvec,
$$
transports an initial tangent vector $\vvec(0)\in T(\x(0))$ to $\vvec(t)\in T(\x(t))$ along the trajectory $\x(t)$ in a parallel manner, i.e., $\nabla_{\dot{\x}}\vvec(t)=0.$ With straightforward calculations as in Proposition \ref{proposition: x on the manifold}, it can be verified that the CSD with the initial condition $\x(0) \in \M, \vvec_i(0)\in T_{\x(0)}\M,i=1,\cdots,k$ satisfies $\x(t) \in \M, \vvec_i(t)\in T_{\x(t)}\M,i=1,\cdots,k$ for $t\geq 0$.

\section{Momentum-Based Constrained Saddle Dynamics}\label{sec: MCSD}

Since CSD can be regarded as a special gradient-based method, it is natural to expect that, as in the gradient descent methods, the convergence rate of discrete CSD is slow when the problem is ill-conditioned (see Section \ref{sec: CSD with exact eigenvector}). This motivates the momentum-based constrained saddle dynamics (MCSD) that is capable of accelerating the convergence. In this section, we introduce MCSD and study its properties.

By introducing the momentum, saddle dynamics with Polyak heavy ball \cite{polyak1964some, nesterov2013introductory,luo2025accelerated} in Euclidean space is defined as 
$$
\ddot{\x}+\gamma_1 \dot{\x}+\gamma_2\left(\I-2\sum_{i=1}^k {\vvec}_i{\vvec}_i^\top\right)\nabla f(\x)=0,
$$
where $\gamma_1,\gamma_2>0$. Accordingly, by replacing the Euclidean acceleration $\ddot{\x}$ with the covariant derivative induced by the Levi-Civita connection, and $\nabla f(\x)$ with $\grad f(\x)$, saddle dynamics with Polyak heavy ball on a manifold is defined as 
\begin{equation}
\frac{D}{dt}\dot{\x}+\gamma_1 \dot{\x}+\gamma_2\left(\I-2\sum_{i=1}^k {\vvec}_i{\vvec}_i^\top\right)\gradient f(\x)=0,
\label{heavy ball second order}
\end{equation}
where $\frac{D}{dt}$ denotes the covariant derivative of a vector field $\zeta(t)$ along a smooth curve $\x(t)\in \M$. Viewing $\zeta(t)\in T_{\x(t)}\M$ as an ambient vector-valued function, $\frac{D}{dt}$ is defined by
$$\frac{D}{dt}\zeta(t)=\Pvec(\x)\frac{d}{dt}\zeta(t)=\Pvec(\x)\dot{\zeta}(t),$$
where $\dot{\zeta}(t)=\frac{d}{dt}\zeta(t)$ is the ordinary derivative in the ambient Euclidean space.

Define the velocity as $\rvec=\dot{\x}\in T_{\x}\M$ and add the $\vvec$-dynamics, the second-order differential equation in \eqref{heavy ball second order} is converted into the following first-order system
\begin{equation}\label{momentum on the manifold}
\left\{
\begin{aligned}
&\dot{\x}=\rvec,\\
&\frac{D}{dt}\rvec=-\gamma_1 \rvec -\gamma_2\left(\I-2\sum_{i=1}^k {\vvec}_i{\vvec}_i^\top\right)\gradient f(\x),\\
&\dot{\vvec}_i=-   \left(\I-{\vvec}_i{\vvec}_i^\top-\sum_{j=1}^{i-1}2{\vvec}_j{\vvec}_j^\top\right)\tilde{\hessian} f(\x) [\vvec_i]\\
&\qquad-\A(\x)(\A(\x)^\top \A(\x))^{-1}(\nabla^2 \Cvec(\x)\dot{\x})^\top \vvec_i,\ i=1,\cdots,k .
\end{aligned}
\right.
\end{equation}

\begin{proposition}\label{proposition: x on the manifold}
    If the initial conditions satisfy $\x(0)\in \M, \rvec(0)\in T_{\x(0)}\M, \vvec_i(0)\in T_{\x(0)}\M, i=1,\cdots,k$, then $\x(t), \rvec(t), \vvec_i(t),i=1,\cdots,k$ induced by \eqref{momentum on the manifold} satisfy $\x(t)\in \M, \rvec(t)\in T_{\x(t)}\M, \vvec_i(t)\in T_{\x(t)}\M, i=1,\cdots,k$ for $t\geq 0$.
\end{proposition}
\begin{proof}
We prove $\vvec_i(t)\in T_{\x(t)}\M,i=1,\cdots,k$ by induction. For $\vvec_1(t)$, we have that $\A(\x(t))^\top \vvec_1(t)$ satisfies the following ODE
    $$
    \begin{aligned}
    &\frac{d}{dt}(\A(\x)^\top\vvec_1)=(\nabla^2 \Cvec(\x)\dot{\x})^\top\vvec_1+\A(\x)^\top\dot{\vvec}_1\\
    &=(\nabla^2 \Cvec(\x)\rvec)^\top\vvec_1-\A(\x)^\top\left(\I-{\vvec}_1{\vvec}_1^\top\right)\tilde{\hessian} f(\x) [\vvec_1]\\
&\qquad-\A(\x)^\top\A(\x)(\A(\x)^\top \A(\x))^{-1}(\nabla^2 \Cvec(\x)\rvec)^\top \vvec_1\\
&= -\A(\x)^\top\tilde{\hessian} f(\x) [\vvec_1]+(\vvec_1^\top\tilde{\hessian} f(\x) [\vvec_1]) \A(\x)^\top \vvec_1\\
&= \vvec_1^\top\tilde{\hessian} f(\x) [\vvec_1] \A(\x)^\top \vvec_1,
    \end{aligned}
    $$
with initial condition $\A(\x(0))^\top\vvec_1(0)=0$ and hence $\A(\x(t))^\top\vvec_1(t)\equiv0$, i.e. $\vvec_1(t)\in T_{\x(t)}\M$, where the last equality follows from that $\tilde{\hessian} f(\x) [\vvec_1]\in T_{\x}\M$. Suppose $\vvec_i(t)\in T_{\x(t)\M},1\leq i\leq i_0-1\leq k-1$, for $\vvec_{i_0}$, we have that
    $$
    \begin{aligned}
    &\frac{d}{dt}(\A(\x)^\top\vvec_{i_0})=(\nabla^2 \Cvec(\x)\dot{\x})^\top\vvec_{i_0}+\A(\x)^\top\dot{\vvec}_{i_0}\\
    &=(\nabla^2 \Cvec(\x)\rvec)^\top\vvec_{i_0}-\A(\x)^\top\left(\I-{\vvec}_{i_0}{\vvec}_{i_0}^\top-\sum_{j=1}^{i_0-1}2{\vvec}_j{\vvec}_j^\top\right)\tilde{\hessian} f(\x) [\vvec_{i_0}]\\
&\qquad-\A(\x)^\top\A(\x)(\A(\x)^\top \A(\x))^{-1}(\nabla^2 \Cvec(\x)\rvec)^\top \vvec_{i_0}\\
&= -\A(\x)^\top\tilde{\hessian} f(\x) [\vvec_{i_0}]+(\vvec_{i_0}^\top\tilde{\hessian} f(\x) [\vvec_{i_0}]) \A(\x)^\top \vvec_{i_0}\\
&= \vvec_{i_0}^\top\tilde{\hessian} f(\x) [\vvec_{i_0}] \A(\x)^\top \vvec_{i_0},
    \end{aligned}
    $$
with initial condition $\A(\x(0))^\top\vvec_{i_0}(0)=0$ and hence $\A(\x(t))^\top\vvec_{i_0}(t)\equiv0$, i.e. $\vvec_{i_0}(t)\in T_{\x(t)}\M$.
Since $$\Pvec(\x)\dot{\rvec}=-\gamma_1 \rvec -\gamma_2\left(\I-2\sum_{i=1}^k {\vvec}_i{\vvec}_i^\top\right)\gradient f(\x),$$
left-multiplying both sides by $\A(\x)^\top$ and using $\A(\x)^\top\Pvec(\x)=0$ and $\A(\x)^\top \grad f(\x)=0$ gives $0=-\gamma_1 \A(\x)^\top \rvec$, i.e., $\rvec(t) \in T_{\x(t)}\M$. Then, $$\frac{d}{dt}\Cvec(\x)=\A(\x)^\top\dot{\x}=\A(\x)^\top \rvec=0$$
with $\Cvec(\x(0))=0$, and hence, $\Cvec(\x(t))\equiv 0$ which means $\x(t)\in \M$ for all $t\geq0$.
\end{proof}

Since the manifold $\M$ is embedded in Euclidean space, we can express the covariant derivative on the manifold in terms of Euclidean derivatives. We need to introduce the Lagrange multiplier to eliminate the component in the normal space. Consider the dynamics of $\rvec$ with Lagrange multiplier
$$
\dot{\rvec}=-\gamma_1 \rvec -\gamma_2\left(\I-2\sum_{i=1}^k {\vvec}_i{\vvec}_i^\top\right)\gradient f(\x)+\A(\x)\vartheta.
$$
To satisfy the constraint $\rvec \in T_\x\M$, $\frac{d}{dt}(\A(\x)^\top \rvec)=\A(\x)^\top \dot{\rvec}+(\nabla^2\Cvec(\x)\dot{\x})^\top\rvec=0$, we should choose  
$$
\vartheta=-(\A(\x)^\top \A(\x))^{-1}(\nabla^2 \Cvec(\x)\dot{\x})^\top \rvec,
$$
and \eqref{momentum on the manifold} is equivalent to the following MCSD defined in the Euclidean space
\begin{equation}\label{dynamics with momentum with Euclidean Language}
\left\{
\begin{aligned}
&\dot{\x}=\rvec,\\
&\dot{\rvec}=-\gamma_1 \rvec -\gamma_2\left(\I-2\sum_{i=1}^k {\vvec}_i{\vvec}_i^\top\right)\gradient f(\x)\\
&\qquad -\A(\x)(\A(\x)^\top \A(\x))^{-1}(\nabla^2 \Cvec(\x)\dot{\x})^\top \rvec,\ i=1,\cdots,k ,\\
&\dot{\vvec}_i=-   \left(\I-{\vvec}_i{\vvec}_i^\top-\sum_{j=1}^{i-1}2{\vvec}_j{\vvec}_j^\top\right)\tilde{\hessian} f(\x) [\vvec_i]\\
&\qquad-\A(\x)(\A(\x)^\top \A(\x))^{-1}(\nabla^2 \Cvec(\x)\dot{\x})^\top \vvec_i,\ i=1,\cdots,k .
\end{aligned}
\right.
\end{equation}
Although MCSD in \eqref{dynamics with momentum with Euclidean Language} keep the constraint $\x(t)\in \M, \rvec(t) \in T_{\x(t)}\M,\vvec_i(t)\in T_{\x(t)}\M, i=1,\cdots,k$, if the initial conditions satisfy the constraint (Proposition \ref{proposition: x on the manifold}), a perturbation may easily deviate $\x$ away from the manifold. To achieve the linear stability, that is, with a small perturbation to $\x^*$ (which deviates $\x^*$ away from the manifold), the dynamics lead $\x(t)$ come back to $\x^*$, an additional stability term, $-\mu \sum_{i=1}^m c_i(\x)\nabla c_i(\x)$ with $\mu>0$, is added to the $\x$-dynamic, which can be seen as the gradient flow of $\|\Cvec(\x)\|_2^2$. This results in the following dynamics
\begin{equation}
\left\{
\begin{aligned}
&\dot{\x}=\rvec-\mu \sum_{i=1}^m c_i(\x)\nabla c_i(\x),\\
&\dot{\rvec}=-\gamma_1 \rvec -\gamma_2\left(\I-2\sum_{i=1}^k {\vvec}_i{\vvec}_i^\top\right)\gradient f(\x)\\
&\qquad -\A(\x)(\A(\x)^\top \A(\x))^{-1}(\nabla^2 \Cvec(\x)\dot{\x})^\top \rvec,\\
&\dot{\vvec}_i=-   \left(\I-{\vvec}_i{\vvec}_i^\top-\sum_{j=1}^{i-1}2{\vvec}_j{\vvec}_j^\top\right)\tilde{\hessian} f(\x) [\vvec_i]\\
&\qquad-\A(\x)(\A(\x)^\top \A(\x))^{-1}(\nabla^2 \Cvec(\x)\dot{\x})^\top \vvec_i,\ i=1,\cdots,k .
\end{aligned}
\right.
\label{modified ACHiSD}
\end{equation}

\begin{remark}
    If $\x \in \M$, that is, $c_i(\x)=0, i=1,\cdots,m$, then \eqref{modified ACHiSD} degenerates to \eqref{dynamics with momentum with Euclidean Language}.
\end{remark}

\begin{proposition}[Local convergence of the MCSD]
    Assume $\x^*\in \M, \{\vvec_i^*\}_1^{d-m}\in T_{\x^*}\M$ satisfies $\|\vvec_i\|=1$ and the eigenvalues of $\hessian f(\x^*)$ are  \begin{equation}
    \lambda_1^*<\cdots<\lambda_k^*<0< \lambda_{k+1}^*\leq \cdots \leq \lambda_{d-m}^*.
    \label{eq: eigengap}
    \end{equation} Then $(\x^*,\rvec^*,\vvec_1^*,\cdots,\vvec_k^*)$ is a linearly stable steady state of \eqref{modified ACHiSD}, if and only if $\x^*$ is an index-$k$ saddle point, $\rvec^*=0$, and $\hessian f(\x^*)[\vvec^*_i]=\lambda_i^*\vvec_i^*, i=1,\cdots,k$.
\end{proposition}
\begin{proof}
To determine the linear stability of \eqref{modified ACHiSD}, we calculate its Jacobian operator as follows
$$
J=\frac{\partial(\dot{\x},\dot{\rvec},\dot{\vvec}_1,\cdots,\dot{\vvec}_k)}{\partial(\x,\rvec,\vvec_1,\cdots,\vvec_k)}=\begin{pmatrix}
    J_{\x \x} & \I & \mathbf{0} & \mathbf{0} & \cdots & \mathbf{0} \\
    J_{\rvec\x} & -\gamma_1\I & J_{\rvec1} & J_{\rvec2} & \cdots & J_{\rvec k} \\
    J_{1\x} & \mathbf{0} & J_{11} & \mathbf{0} & \cdots & \mathbf{0} \\
    J_{2\x} & \mathbf{0} & J_{21} & J_{22} & \cdots & \mathbf{0} \\
    \vdots & \vdots &\vdots &\vdots &\vdots &\vdots \\
    J_{k\x} & \mathbf{0} & J_{k1} & J_{k2} & \cdots & J_{kk}
\end{pmatrix},
$$
where 
$$
\begin{aligned}
&J_{\x\x}=\frac{\partial \dot{\x}}{\partial \x}=-\mu \sum_{i=1}^m (c_l(\x)\nabla^2c_l(\x)+\nabla c_l(\x)\nabla c_l(\x)^\top), \\
&J_{\rvec\x}=\frac{\partial \dot{\rvec}}{\partial \x}=-\gamma_2\left(\I-2\sum_{i=1}^k {\vvec}_i{\vvec}_i^\top\right)(\mathbf{H}(\x)-\A(\x)(\A(\x)\A(\x))^{-1}(\nabla^2 \Cvec (\x) \gradient f(\x))^\top)\\
&\qquad \qquad\qquad- \partial_\x (\A(\x)(\A(\x)^\top \A(\x))^{-1}(\nabla^2 \Cvec(\x)\dot{\x})^\top)\rvec,\\
& J_{\rvec i}= \frac{\partial \dot{\rvec}}{\partial \vvec_i}= 2\gamma_2 (\vvec_i^\top \gradient f(\x) \I+\vvec_i \gradient f(\x)^\top),\\
& J_{ii}=\frac{\partial \dot{\vvec}_i}{\partial \vvec_i}=-\left(\I-2\sum_{j=1}^i {\vvec}_j{\vvec}_j^\top\right)\tilde{\hessian}f(\x)+\langle 
\vvec_i, \tilde{\hessian}f(\x)[\vvec_i] \rangle \I \\
&\qquad \qquad\qquad-\A(\x)(\A(\x)^\top \A(\x))^{-1} (\nabla^2 \Cvec(\x)\dot{\x})^\top,
\end{aligned}
$$
with
$$
\mathbf{H}(\x)=\Pvec(\x)(\nabla^2 f(\x)-\nabla^2 \Cvec(\x)(\A(\x)\A(\x))^{-1}\A(\x)^\top\nabla f(\x)).
$$
In the following, $J(\x^*, \rvec^*, \vvec_1^*, \cdots , \vvec^*_k)$ is denoted as $J^*$ and the blocks of $J^*$ are denoted with corresponding subscripts.

``$\Longleftarrow$": Supposing that $\x^*$ is an index-$k$ saddle point, $\rvec^*=0$, and $\hessian f(\x^*)[\vvec_i^*]=\lambda_i^*\vvec_i^*, i = 1, \cdots, k$, we  have $\gradient f(\x^*) = 0$, and $J_{\rvec i}^*=\mathbf{0}$,
$$
\begin{aligned}
& J_{\x\x}^*= -\mu \sum_{i=1}^m \nabla c_l(\x^*)\nabla c_l(\x^*)^\top,  \\
&J_{\rvec\x}^*=-\gamma_2\left(\I-2\sum_{i=1}^k {\vvec}_i^*{\vvec_i^*}^\top\right)\mathbf{H}(\x^*),\\
&J_{ii}^*=-\left(\I-2\sum_{j=1}^i {\vvec}_i^*{\vvec_i^*}^\top\right)\tilde{\hessian}f(\x^*)+\lambda_i^*\I.\\
\end{aligned}
$$ 
Hence, $J^*$ is in the form of 
\begin{equation}
J^*=\begin{pmatrix}
    D_1 & \mathbf{0} \\
    J_{\vvec\x}^* & D_2
\end{pmatrix}, D_1=\begin{pmatrix}
    J_{\x\x}^* & \mathbf{I}\\
    J_{\rvec \x}^* & -\gamma_1\I
\end{pmatrix}
, D_2=\begin{pmatrix}
     J_{11}^* & \mathbf{0} & \cdots & \mathbf{0} \\
     J_{21}^*  & J_{22}^*  & \cdots & \mathbf{0} \\
    \vdots &\vdots &\vdots &\vdots \\
    J_{k1}^*  & J_{k2}^*  & \cdots & J_{kk}^*
\end{pmatrix}.
\label{eq: diag of J star}
\end{equation}
Therefore the eigenvalues of $J^*$ is completely determined by that of $D_1$ and $D_2$, and the eigenvalues of $D_2$ is completely determined by that of $J_{ii}^*$. Note that
\begin{equation}
J_{ii}^*\nabla c_j(\x^*)\!=\!\lambda_i^*\nabla c_j(\x^*), 1\!\leq\! j\leq m,\,
J_{ii}^*\vvec_j^*\!=\!\begin{cases}
(\lambda_i^*+\lambda_j^*)\vvec_j^*, 1\!\leq\! j\!\leq\! i,\\
(\lambda_i^*-\lambda_j^*)\vvec_j^*,i\!<\!j\!\leq \!d-m;
\end{cases} 
\label{eq: eigen Jii}
\end{equation}
which means the $d$ eigenvalues of $J_{ii}^*$ are $\lambda_i^*+\lambda_1^*, \cdots,\lambda_i^*+\lambda_i^*,\lambda_i^*-\lambda^*_{i+1},\cdots,\lambda_i^*-\lambda_{d-m}^*$ and $\lambda_i^*$ with multiplicity $m$, and all of them are negative. Now, we only need to consider the eigenvalues of $D_1$. Since $J_{\x \x}^*$ is a symmetric matrix whose eigenvalues are non-positive with $m$ negative eigenvalues and $d-m$ zero eigenvalues, we assume that its spectrum decomposition is $J_{\x \x}^*= V_c\mathrm{diag}(-\tau_1,\cdots,-\tau_m)V_c^\top$, where $V_c$ is an orthogonal matrix whose column is the linear composition of $\nabla c_l(\x^*)$ and $\tau_i>0,i=1,\cdots,m$. Then, $Q=(V_c,R^*)$ with $R^*=(\vvec_1^*,\cdots,\vvec_k^*)$, is also an orthogonal matrix. This is because, $\vvec_i^*\in T_{\x^*}\M$ for $1\leq i\leq k$, so we have $\nabla c_l(\x^*)^\top \vvec_i^*=0$, for $ 1\leq l\leq m$. We have that
$$
\tilde{D}_1=\begin{pmatrix}
    Q & \mathbf{0}\\
    \mathbf{0}  & Q
\end{pmatrix}^\top
D_1\begin{pmatrix}
    Q & \mathbf{0}\\
    \mathbf{0}  & Q
\end{pmatrix}=\begin{pmatrix}
    Q^\top J_{\x\x}^* Q & \mathbf{I}\\
    Q^\top J_{\rvec\x}^* Q  & -\gamma_1 \I
\end{pmatrix}.
$$
Since $J_{\x \x}^* R^*=\mathbf{0}, V_c^\top J_{\rvec\x}^*=\mathbf{0}$, and
$$
\vvec_i^\top J_{\rvec \x}^* \vvec_j=\begin{cases}
    0, i\neq j,\\
\gamma_2 \lambda_i,i=j\leq k,\\
    -\gamma_2\lambda_i,i=j>k,
\end{cases}\; \text{ for } i,j=1,\cdots,n-m,  
$$
we have that $$Q^\top J_{\x\x}^* Q=\text{diag}(-\tau_1,\cdots,-\tau_m,\underbrace{0,\cdots,0}_{n-m}),$$ 
$$Q^\top J_{\rvec\x}^* Q=\text{diag}(\underbrace{0,\cdots,0}_{m},\gamma_2\lambda_1^*,\cdots,\gamma_2\lambda_k^*,-\gamma_2\lambda_{k+1}^*,\cdots,-\gamma_2\lambda_{n-m}^*).$$
After a series of permutations, $\tilde{D}_1$ is similar to a block upper triangular matrix whose diagonal blocks are 
\begin{equation}
\{-\tau_j\}_{j=1}^m , -\gamma_1\I_m, D_{2\times2}^1,\cdots,D_{2\times2}^{d-m}, D_{2\times2}^i=\begin{pmatrix}
    0 & 1 \\ -\gamma_2|\lambda_i^*| & -\gamma_1
\end{pmatrix}, 1\leq
i\leq d-m,
\label{eq: eigen D1}
\end{equation}
The eigenvalues of $D_{2\times2}^i$ are $\frac{-\gamma_1\pm \sqrt{\gamma_1^2-4\gamma_2|\lambda_i^*|
}}{2}$. If $\gamma_1^2-4\gamma_2|\lambda_i^*|\leq 0$, the real parts of the eigenvalues are negative; If $\gamma_1^2-4\gamma_2|\lambda_i^*|> 0$, the eigenvalues are negative real numbers because $\frac{-\gamma_1\pm\sqrt{\gamma_1^2-4\gamma_2|\lambda_i^*|}}{2}\leq \frac{-\gamma_1+\sqrt{\gamma_1^2-4\gamma_2|\lambda_i^*|}}{2}<\frac{-\gamma_1+\gamma_1}{2}=0$. In conclusion, the real parts of all eigenvalues of $J^*$ are negative, which indicates that $(\x^*,\rvec^*=0,\vvec_1^*,\cdots,\vvec_k^*)$ is a linearly stable steady state of \eqref{modified ACHiSD}.

``$\Longrightarrow$": Supposing that $(\x^*,\rvec^*,\vvec_1^*,\cdots,\vvec_k^*)$ is a linearly stable steady state, which means $\dot{\x}=\dot{\rvec}=\dot{\vvec}_i=0$, indicating 
\begin{equation}
\rvec^*=0, \left(\I-2\sum_{j=1}^i {\vvec}_j^*{{\vvec}_j^*}^\top\right)\tilde{\hessian}f(\x^*)[\vvec_i^*]=\mu_i^*\vvec_i^*, \mu_i^*=\langle \vvec_i^*,\tilde{\hessian}f(\x^*)[\vvec_i^*] \rangle
\label{eq: equilibrium}
\end{equation}
We now show by induction that for $i = 1, \cdots, k$,
\begin{equation}
\tilde{\hessian}f(\x^*)[\vvec_i^*]=\mu_i^*\vvec_i^*\neq0, \langle \vvec_j^*,\vvec_i^* \rangle=\delta_{ij}, j=1,\cdots,i-1.
\label{eq: eigen}
\end{equation}
The $i = 1$ case is obtained from \eqref{eq: equilibrium} directly. Assumed that \eqref{eq: eigen} holds for $1 \leq i < l$, by taking $i = l$ in \eqref{eq: equilibrium} we have,
$$
\left(\tilde{\hessian}f(\x^*)-\sum_{j=1}^{l-1}2\mu_j\vvec_j^*\vvec_j^*\right)\vvec_l^*=\mu_l^*\vvec_l^*.
$$
Since $\tilde{\hessian}f(\x^*)$ and $\tilde{\hessian}f(\x^*)-\sum_{j=1}^{l-1}2\mu_j\vvec_j^*\vvec_j^*$ share the same eigenvectors, $\vvec_l^*$ is also an eigenvector of $\tilde{\hessian}f(\x^*)$ with an eigenvalue $\mu_l^*$, and hence $$\left(\sum_{j=1}^{l-1}2\mu_j\vvec_j^*{\vvec_j^*}^\top\right)\vvec_l^*=0$$ leads to $\langle \vvec_l^*,\vvec_j^* \rangle=0$, which completes the induction. As a result of \eqref{eq: eigen}, $\I-2\sum_{i=1}^k \vvec_i^*\vvec_i^*$ is an orthogonal matrix (so that it is full rank), which, together with $\dot{\rvec}=0,\rvec^*=0$, indicates $\gradient f(\x^*) = 0$, i.e. $\x^*$ is a critical point. 

Now, we only need to show that $\x^*$ is an index-$k$ saddle point. Recall that $\{(\mu_i^*,\vvec_i^*)\}_1^{d-m}$ are eigenpairs of $\hessian f(\x^*)$. Similarly to \eqref{eq: diag of J star} and \eqref{eq: eigen D1}, the $2d$  eigenvalues of $D_1$ are $\{-\tau_i\}_{i=1}^m$, $-\gamma_1$ (with multiplicity $m$), and the residual $2d-2m$ eigenvalues are
\begin{equation}
\label{eigenvalue x and r}
\frac{-\gamma_1\pm \sqrt{\gamma_1^2+4\gamma_2\mu_1^*}}{2}, \cdots, \frac{-\gamma_1\pm \sqrt{\gamma_1^2-4\gamma_2\mu^*_{d-m}}}{2},
\end{equation}
which have negative real parts from the linear stability, indicating that $\mu_i^*<0,i\leq k$, $\mu_i^*>0,i> k$, that is, $\x^*$ is an index-k saddle point. Finally, it follows from \eqref{eq: eigen Jii} that the eigenvalues of $J_{ii}^*$ are
$$
\mu_i^*+\mu_1^*,\cdots,\mu_i^*+\mu_i^*, \mu_i^*-\mu_{i+1}^*, \cdots, \mu_i^*-\mu_{d-m}^*, \mu_i^* (\text{with multiplicity m}).
$$
To ensure that all eigenvalues are negative, we have $\mu_1^*<\mu_2^*<\cdots<\mu_k^*$, which means $\mu_i^*=\lambda_i^*$. This completes the proof.
\end{proof}

\begin{remark}
    In \cite[Theorem 1]{yin2020constrained}, the authors have analyzed the linear stability of the CSD and obtained a convergence rate of $\e^{-\mu t}$ with respect to $\x$. Choosing $\gamma_2=1$ and the critical damping parameter $\gamma_1=2\sqrt{\mu}$, \eqref{eigenvalue x and r} yields a convergence rate of $\e^{-\sqrt{\mu} t}$. Consequently, in directions corresponding to small eigenvalues ($\mu\ll 1$), the momentum term leads to a significant acceleration of convergence. The analysis presented later on the discrete dynamics further demonstrates the acceleration effects in the case of an ill-conditioned Hessian. 
\end{remark}

\begin{remark}
The simplicity of the negative eigenvalues, that is, $\lambda_1^*<\cdots<\lambda_k^*<0< \lambda_{k+1}^*\leq \cdots \leq \lambda_{d-m}^*$, is required to ensure the convergence of the unstable eigenvectors. Indeed, when the negative eigenvalues are not simple, the associated eigenvectors are not uniquely defined and may rotate within the corresponding eigenspace. Nevertheless, one can prove that local convergence still holds even when the negative eigenvalues have multiplicity, by adapting the arguments in Section \ref{sec: one step v update}. In this case, instead of guaranteeing convergence of the individual unstable eigenvectors, we should establish convergence of the associated Householder operator $R=\I-2\sum_{i=1}^k \vvec_i{\vvec_i}^\top$, which is invariant under orthogonal changes of basis within the unstable eigenspace, and therefore remains uniquely defined.
\end{remark}

\section{Discrete constrained saddle dynamics and momentum-based constrained saddle dynamics}\label{sec: CSD with exact eigenvector}
Although Section \ref{sec: MCSD} shows that the continuous-time CSD and MCSD exhibit local convergence with rates $e^{-\mu t}$ and $e^{-\sqrt{\mu} t}$, respectively, practical implementations necessarily rely on discrete-time schemes with finite step sizes. Hence, the convergence rate depends not only on the smallest eigenvalue $\mu$, but also on the largest eigenvalue $L$, since $L$ constrains the admissible step size required for numerical stability. Consequently, the discrete-time convergence rate becomes condition-number dependent. Motivated by this observation, we study the convergence of the discrete CSD and MCSD with the exact unstable eigenvectors in this section. 

\subsection{Convergence analysis of the discrete constrained saddle dynamics}\label{subsec: convergence CSD}
When analyzing the convergence of the discrete saddle dynamics (SD), the saddle dynamics are often implemented under the framework of the following algorithm, with the exact eigenvector assumption \cite{luo2022sinum,su2025improved,luo2025accelerated}
\begin{equation}\label{eq: saddle algorithm}
\begin{aligned}
    &\x_{n+1}=\mathrm{Ret}_{\x_n}\left(-\Delta t\left(\I-2\sum_{i=1}^k\vvec^i_n{\vvec^i_n}^\top\right)\gradient  f(\x_n)\right),\\
    & \{\vvec^i_{n+1}\}_{i=1}^k=\text{EigenSolver}(\{\vvec^i_n\}_{i=1}^k,\hessian f(\x_{n+1})), 
\end{aligned}
\end{equation}
where $\Delta t$ is the step size, $\mathrm{Ret}_\x(\vvec)$ denotes a retraction, i.e., a mapping $\mathrm{Ret}_\x: T_\x\M\rightarrow \M$ that sends a tangent vector to a point on the manifold and satisfies $$\mathrm{Ret}_\x(\mathbf{0})=\x, \left.\frac{d}{dt}\right|_{t=0}\mathrm{Ret}_\x(t\eta)=\eta,\forall \eta\in T_{\x}\M.$$ 

\begin{remark}
    In \eqref{eq: saddle algorithm}, instead of updating the unstable directions $\vvec_i$ via an explicit Euler discretization of the $\vvec$-dynamics in \eqref{eq: SD}, the algorithm computes the unstable eigenspace of $\hessian f(\x_{n+1})$ exactly at each iteration. While this assumption simplifies the convergence analysis, it requires solving an eigenvalue problem at every step, which can be computationally expensive. In this section, we adopt the exact-eigenvector assumption for brevity. Nevertheless, in Section \ref{sec: one step v update}, we will show that the local convergence results still hold with the Euler discretization of $\vvec$-dynamics.
\end{remark}
We focus on the convergence to a strict index-$k$ saddle point $\x^*$, i.e., $\hessian f(\x^*)$ has no zero eigenvalues, which is specified in the following assumption.

\begin{assumption}[Strict saddle]\label{assumption: Hessian}
There exist $\delta>0, L>\mu>0$ such that if $\x\in U_\delta=\{\x\in \M,d(\x,\x^*)< \delta\}$, where $d(\x,\x^*)$ is the geodesic distance, the eigenvalues of $\hessian f(\x)$ satisfy
\begin{equation}
    -L\leq\lambda_1\leq \cdots\leq\lambda_k\leq -\mu < 0< \mu\leq\lambda_{k+1} \leq \cdots \leq \lambda_{d-m} \leq L.
    \label{eq:eigenvalues}
\end{equation}
\end{assumption}

\begin{remark}
    Assumption \ref{assumption: Hessian} holds for every strict index-$k$ saddle point $\x^*$. $L$ and $\mu$ can be regarded as upper and lower bounds of the absolute eigenvalues of $\hessian f(\x), \x \in U_\delta$. Thus, $\kappa=L/\mu$ represents the condition number of $\hessian f(\x),\x\in U_\delta.$
\end{remark} 

When analyzing local convergence, it is convenient to use local Taylor expansions. Since $\{\vvec^i(\x)\}_{i=1}^{d-m}$ is an eigenframe of $\hessian f(\x)$ and thus depends on $\x$, it is necessary to first characterize the smoothness of the Householder operator $R(\x)=\I-2\sum_{i=1}^k\vvec^i(\x){\vvec^i(\x)}^\top$ as a function of $\x$, which is specified in the following lemma.

\begin{lemma}[The smoothness of the reflection operator]\label{lem:Riesz-proj-smooth}
Suppose Assumption \ref{assumption: Hessian} holds. Let $\Gamma\subset \mathbb{C}$ be a positively oriented simple closed $C^1$-curve enclosing the real interval $[-L,-\mu]$ but does not enclose $[\mu,L]$. For $\x\in U_\delta$, define
\begin{equation}\label{eq:Riesz-Px-def}
  R(\x)
  := \I+\frac{1}{\pi i}\oint_\Gamma \big(\hessian f(\x)-z\I\big)^{-1}\,\mathrm dz
  :T_x\M\to T_x\M.
\end{equation}
Then 
\begin{equation}\label{eq: projection operator}
R(\x)=\I-2\sum_{i=1}^k\vvec^i(\x){\vvec^i(\x)}^\top,
\end{equation}
where $\{\vvec^1(\x),\dots,\vvec^k(\x)\}$ is the local orthonormal frame of the negative eigenspace of $\hessian f(\x)$. Consequently, the Householder operator $R(\x)$ is smooth on $U_\delta$. 
\end{lemma}

\begin{proof}
Fix $\bar{\x}\in U_\delta$, choose an open neighbourhood $U_0\subset U_\delta$ of $\bar{\x}$ and a smooth orthonormal frame 
$\{e_1(\x),\dots,e_{d-m}(\x)\}$ of $T_{\x}\M$ on $U_0$.
With respect to this frame, the Riemannian Hessian is represented by a symmetric real $(d-m)\times (d-m)$ matrix
\[
  H(\x) = \big\{h_{ij}(\x)\big\}, \hessian f(\x)[e_j(\x)] = \sum_{i=1}^{d-m} h_{ij}(\x)\,e_i(\x).
\]
Since $f$ is smooth on $\M$, each $h_{ij}(\x)$
is also smooth on $\x\in U_0$.
In this frame, the Householder operator $R(\x)$ defined in \eqref{eq:Riesz-Px-def} is represented by
\[
  R_{mat}(\x) = \I_{d-m}+\frac{1}{\pi i}\oint_\Gamma \big(H(\x)-z\I_{d-m}\big)^{-1}\,\mathrm dz.
\]
Fix $\x\in U_0$, since $H(\x)$ is a real symmetric matrix, there exists an orthonormal basis 
$\{\mathbf{u}_1(\x),\dots,\mathbf{u}_{d-m}(\x)\}$ consisting of eigenvectors of $H(\x)$, with corresponding eigenvalues of $\hessian f(\x)$,
$\lambda_1(\x),\dots,\lambda_{d-m}(\x)$. In this eigenbasis,
\[
  H(\x) = \sum_{i=1}^{d-m} \lambda_i(\x)\,\mathbf{u}_i(\x) \mathbf{u}_i(\x)^\top,\]
  and
  \[
  (H(\x)-z\I_{d-m})^{-1}
  = \sum_{i=1}^{d-m} \frac{1}{\lambda_i(\x)-z}\,\mathbf{u}_i(\x) \mathbf{u}_i(\x)^\top.
\]
Note that if $z\in\Gamma\subset \mathbb{C}$, from the definition of $\Gamma$, we have $z\notin\sigma(H(\x))$. Then,
$$
  R_{mat}(\x)=\I_{d-m}+ \sum_{j=1}^{d-m}
     \Bigg(\frac{1}{\pi i}\oint_\Gamma \frac{1}{\lambda_j(\x)-z}\,\mathrm dz\Bigg) \mathbf{u}_j(\x) \mathbf{u}_j(\x)^\top.
$$
From the residue theorem, if $\lambda_j(\x)$ lies inside $\Gamma$, then $\frac{1}{2\pi i}\oint_\Gamma \frac{1}{\lambda_j(\x)}-z\,\mathrm dz = -1,$
while the integral vanishes if $\lambda_j(\x)$  lies outside $\Gamma$. Since $\Gamma$ encloses precisely the negative eigenvalues of $H(\x)$ and none of the nonnegative ones. Hence
\[
  R_{mat}(\x)
  =\I_{d-m} -2\sum_{i=1}^{k} \mathbf{u}_i(\x) \mathbf{u}_i(\x)^\top.
\]
The right-hand side is exactly the Householder matrix onto the direct sum of the negative eigenspaces of $H(\x)$, which proves \eqref{eq: projection operator}. 

To prove that $R(\x)$ is smooth in $U_\delta$, we only need to prove that $R_{mat}(\x)$ is smooth in $U_0$. Fix $z\in\Gamma$ and define $B(\x):=H(\x)-z\I_{d-m}$. By the spectral gap between the positive and negative eigenvalues and the choice of $\Gamma$, we have that $B(\x)$ is invertible for all $x\in U_0$. Then $B(\x)^{-1} = \frac{\operatorname{adj}(B(\x))}{\det (B(\x))}$, which means that $B(\x)^{-1}$ is smooth on $U_0$. Since $\Gamma$ is compact,
$$
R_{mat}(\x)
  = \I_{d-m}+\frac{1}{\pi i}\oint_\Gamma \big(H(\x)-z\I_{d-m}\big)^{-1} \mathrm{d}z
$$
is smooth, which completes the proof.
\end{proof}

By definition, a retraction is a local diffeomorphism from the tangent space to the manifold in a neighborhood of the zero vector \cite{boumal2023introduction}, as stated in the following lemma.
\begin{lemma}[Local diffeomorphism property of the retraction]\label{local diffeomorphism}
Let $\mathrm{Ret}_\x: T_\x\M\to M$ be the retraction map. There exist $r>0$ and $\delta_1>0$ such that on
\[
  B_r(0_{\x^*})\!:=\!\{\rvec\in T_{\x^*}\M,\|\rvec\|_{\x^*}<r\},\] 
$\mathrm{Ret}_{\x^*}$$:B_r(0_{\x^*})\to \mathrm{Ret}_{\x^*}(B_r(0_{\x^*}))$ is a diffeomorphism, and $U_{\delta_1}$$\subset\mathrm{Ret}_{\x^*}(B_r(0_{\x^*}))\subset U_\delta$. Moreover, there exist positive constants $C_r>0$ and $l>0$ such that
for all $\x\in U_{\delta_1}$ and all $\vvec\in T_\x\M$ with $\|\vvec\|_\x\leq l$, we have
\[
d(\x,\mathrm{Ret}_{\x}(\vvec)) \leq C_r \|\vvec\|_\x.
\]
\end{lemma}

Lemma \ref{local diffeomorphism} tells us that the iterates can be defined using local coordinates on the tangent space. Lemma \ref{lem:Riesz-proj-smooth} allows us to carry out a local Taylor expansion. Therefore, by collecting the first-order information of the update in these local coordinates, we can establish local convergence, as stated in the following proposition.

\begin{proposition}[Local convergence with exact eigenvectors]\label{convergence analysis normal case}
    Suppose Assumption \ref{assumption: Hessian} holds.  There exists a neighborhood of $\x^*$ such that, if the initial condition $\x(0)$ lies in this neighborhood and $\Delta t\leq 1/L$, then the iteration point $\x(n)$ converges to $\x^*$ with a linear convergence rate that depends on the condition number $\kappa=L/\mu$. 
\end{proposition}
\begin{proof}
    From Lemma \ref{local diffeomorphism}, we can choose a radius $r >0$ such that $\mathrm{Ret}_{\x^*}$
is a diffeomorphism from $B_r(0_{\x^*})$ to $ \mathrm{Ret}_{\x^*}(B_r(0_{\x^*}))\subset U_\delta$. From Assumption \ref{assumption: Hessian}, if $\x\in  U_\delta$, $$\|\gradient f(\x)-P_{\x^*\rightarrow\x}\grad f(\x^*)\|=\|\gradient f(\x)\|\leq Ld(\x,\x^*),$$ where $P_{\x^*\rightarrow\x}$ denotes the parallel translation from $\x^*$ to $\x$. Then, if $\x_n\in  U_{r_1}$ and $r_1<\min(\delta_1,\frac{l}{L\Delta t},\frac{\delta_1}{C_r\Delta t L+1})$, we have that
$$
\begin{aligned}
&d(\x_{n+1},\x_n)=d\left(\mathrm{Ret}_{\x_n}\left(-\Delta t\left(\I-2\sum_{i=1}^k \vvec^i_n{\vvec_n^i}^\top\right)\gradient f(\x_n)\right),\x_n\right)\\
&\leq C_r\Delta t\left\|\left(\I-2\sum_{i=1}^k \vvec^i_n{\vvec_n^i}^\top\right)\gradient f(\x_n)\right\|\leq C_r\Delta t L d(\x_n,\x^*)\leq C_r\Delta t L r_1,
\end{aligned}
$$
Then,
$$
d(\x_{n+1},\x^*)\leq d(\x_{n+1},\x_n)+d(\x^*,\x_n)\leq C_r\Delta t L r_1+r_1 \leq \delta_1.
$$
From Lemma \ref{local diffeomorphism}, $\x_{n+1}\in U_{\delta_1} \subset \mathrm{Ret}_{\x^*}(B_r (0_{\x^*}))$. The vectors $\zeta_{n}, \zeta_{n+1}\in T_{\x^*}\M$ are thus well defined by
\begin{equation}
\label{eq: zeta defn}
\zeta_n=\mathrm{Ret}_{\x^*}^{-1}(\x_n), \; \zeta_{n+1}=\mathrm{Ret}_{\x^*}^{-1}(\x_{n+1}),
\end{equation}
which can be regarded as local coordinates of $\x_n,\x_{n+1}$ in 
$T_{\x^*}\M$. Consider the following map restricted to $B_{\delta_1}(0_{x^*})\subset
T_{\x^*}\M:$
$$
F=\mathrm{Ret}_{\x^*}^{-1}\circ G \circ \mathrm{Ret}_{\x^*}: T_{\x^*}\M \rightarrow T_{\x^*}\M,
$$
with $$G(\x)=\mathrm{Ret}_{\x}\left(-\Delta t\left(\I-2\sum_{i=1}^k \vvec^i(\x){\vvec^i(\x)}^\top\right)\gradient f(\x)\right),$$
which translates the original iteration from $\x_n$
to $\x_{n+1}$ into the iteration
$
\zeta_{n+1}=F(\zeta_n).
$
From Lemma \ref{lem:Riesz-proj-smooth}, we know that $G$ is smooth with respect to $\x$ and $F$ is smooth with respect to $\zeta$, we can use a standard
Taylor expansion on $T_{\x^*}\M$ to claim that
\begin{equation}\label{eq: iteration in the tangential space}
F(\zeta_n)= D_\zeta F(0)[\zeta_n]+E(\zeta_n),
\end{equation}
where $D_\zeta F(0)[\zeta_n]=\lim_{t\rightarrow0}\frac{F(t\zeta_n)-F(0)}{t}$ denotes the Fr\'echet derivative and \(\|E(\zeta_n)\|\leq C_{1,r}\|\zeta_n\|^2 \). Define 
$$\Psi(\x,\eta)=\mathrm{Ret}_{\x}(\eta),$$
we have $$G(\x)=\Psi(\x,\eta(\x)), \;
\eta(\x)=-\Delta t\left(\I-2\sum_{i=1}^k \vvec^i(\x){\vvec^i(\x)}^\top\right)\gradient f(\x).$$
and 
\begin{equation}\label{eq: first order term}
\begin{aligned}
&D_\zeta F(0)=D_{\mathbf{y}}\mathrm{Ret}_{\x^*}^{-1}(\mathbf{y})|_{\mathbf{y}=\x^*}\circ D_\x G(\x)|_{\x=\x^*}\circ D_\zeta\mathrm{Ret}_{\x^*}(\zeta)|_{\zeta=0}
\\
&=\I \circ D_\x\Psi(\x,\eta(\x))|_{\x=\x^*} \circ \I\\
&=D_\x \Psi(\x,0)|_{\x=\x^*}+D_\eta\Psi(\x^*,\eta)|_{\eta=0}[D_\x\eta(\x)|_{\x=\x^*}]\\
&=\I-\left.\Delta t D_\x\left(\I-2\sum_{i=1}^k \vvec^i(\x){\vvec^i(\x)}^\top\right)\right|_{\x=\x^*}\underbrace{\grad(\x^*)}_{=0}\\
&\quad-\Delta t \left(\I-2\sum_{i=1}^k \vvec^i(\x^*){\vvec^i(\x^*)}^\top\right)D_\x \grad(\x)|_{\x=\x^*}\\
&=\I-\Delta t \left(\I-2\sum_{i=1}^k \vvec^i(\x^*){\vvec^i(\x^*)}^\top\right)\hessian f(\x^*)=\I-\Delta t R^* H^*,
\end{aligned}
\end{equation}
where $R^*=\I-2\sum_{i=1}^k \vvec^i(\x^*){\vvec^i(\x^*)}^\top$, and $H^*=\hessian f(\x^*)$. By Assumption \ref{assumption: Hessian}, if $\Delta t\leq 1/L$, then $\|\I-\Delta t R^* H^*\|\leq 1-\Delta t \mu<1$. Putting this together with $\zeta_{n+1}=F(\zeta_n)$, \eqref{eq: iteration in the tangential space}, and \eqref{eq: first order term}, we have that 
\begin{equation}\label{eq: decay}
\begin{aligned}
\|\zeta_{n+1}\|&\leq  \left\|\I-\Delta t \left(\I-2\sum_{i=1}^k \vvec^i(\x^*){\vvec^i(\x^*)}^\top\right)\hessian f(\x^*)\right\|\|\zeta_n\|+C_{1,r}\|\zeta_n\|^2\\
&\leq (1-\Delta t \mu)\|\zeta_{n}\|+C_{1,r}\|\zeta_n\|^2.
\end{aligned}
\end{equation}
Note that \eqref{eq: decay} holds for every $\zeta_n\in B_{r_2}(0_{\x^*})$, where $r_2>0$ is a fixed constant that satisfies $\mathrm{Ret}_{\x^*}(B_{r_2}(0_{\x^*}))\subset U_{r_1}$. We further choose $r_2^*\leq \min(r_2,\frac{\Delta t\mu}{2C_{1,r}})$. Then, if $\x_0\in \mathrm{Ret}_{\x^*}(B_{r_2^*}(0_{\x^*}))$, equivalently \(\|\zeta_0\|\leq r_2^*\leq\frac{\Delta t \mu}{2C_{1,r}}\), it follows from \eqref{eq: decay} that $\|\zeta_1\|\leq (1-\frac{\Delta t \mu}{2})\|\zeta_0\|$. Consequently, $\x_1\in \mathrm{Ret}_{\x^*}(B_{r_2^*}(0_{\x^*}))$. By induction, $\|\zeta_{n+1}\|\leq (1-\frac{\Delta t \mu}{2})\|\zeta_n\|$ holds for all subsequent iterates, which leads to a linear convergence with the rate $1-\frac{\Delta t \mu}{2}\in ( 1-\frac{\mu}{2L},1)$ that depends on the condition number $\kappa=L/\mu$.
\end{proof}
In practical implementations, we usually do not expect to have exact eigenvectors due to the finite tolerance of the iteration methods in EigenSolver and the round-off error. Thus, in the Proposition \ref{convergence with inexact} presented below, we assume that $\x(n)$ is updated by  
$$   \x_{n+1}=\mathrm{Ret}_{\x_n}\left(-\Delta t\left(\I-2\hat{P}_n
\right)\gradient  f(\x_n)\right) 
$$
with a small error $\theta$ with respect to the projection onto the eigenspace, that is, 
$$\left\Vert
\hat{P}_n-\sum_{i=1}^k\vvec^i_n\vvec^i_n \text{} ^\top  \right\Vert\leq \theta \ll 1.$$
\begin{proposition}[Local convergence with inexact unstable directions]\label{convergence with inexact}
    Suppose Assumption \ref{assumption: Hessian} holds and that $\theta < \frac{\mu}{2L}$. There exists a neighborhood of $\x^*$ such that, if the initial point lies in this 
neighborhood and $\Delta t\leq 1/L$, the iterates $\x(n)$ converge to $\x^*$ with a linear convergence rate, which depends on the condition number $\kappa=L/\mu$ and on the error $\theta$ of the unstable directions.
\end{proposition}
\begin{proof}
Similar with the proof of Proposition \ref{convergence analysis normal case}, if $r_1<\min(\delta_1,\frac{l}{L\Delta t},\frac{\delta_1}{C_r\Delta t L+1})$ and $\x_n\in  U_{r_1}$, from
$$\x_{n+1}=\mathrm{Ret}_{\x_n}\left(-\Delta t\left(\I-2
\hat{P}_n
\right)\gradient  f(\x_n)\right),$$
we have that 
$$
d(\x_{n+1},\x_n)
\leq C_r\Delta t \|\gradient f(\x_n)  \|\leq C_r\Delta t L d(\x_n,\x^*), d(\x_{n+1},\x^*) \leq \delta_1.
$$
Using the coordinates introduced by \eqref{eq: zeta defn}, and
considering the following maps restricted to $T_{\x^*}\M:$
$$
F=\mathrm{Ret}_{\x^*}^{-1}\circ \mathrm{Ret}_{\x}\left(-\Delta t\left(\I-2\sum_{i=1}^k \vvec^i(\x){\vvec^i(\x)}^\top\right)\gradient f(\x)\right) \circ \mathrm{Ret}_{\x^*},
$$
and 
$$
F_2=\mathrm{Ret}_{\x^*}^{-1}\circ \mathrm{Ret}_{\x}\left(-\Delta t\left(\I-2\sum_{i=1}^k \vvec^i(\x){\vvec^i(\x)}^\top-2V_{err}\right)\gradient f(\x)\right) \circ \mathrm{Ret}_{\x^*},
$$
we get $\zeta_{n+1}=F_2(\zeta_n)$ with  $\|V_{err}\|=\|
\hat{P}_n-\sum_{i=1}^k\vvec_n^i{\vvec_n^i}^\top\|\leq \theta.$ We thus have
$$
\|F(\zeta_n)-F_2(\zeta_n)\|\leq 2 \Delta t \theta L\|\zeta_n\|+C_{2,r}\|\zeta_n\|^2.  
$$
From Proposition \ref{convergence analysis normal case}, if $\Delta t\leq 1/L$, we have that $$\|F(\zeta_n)\|\leq (1-\Delta t \mu)\|\zeta_n\|+C_{1,r}\|\zeta_n\|^2,$$ and hence
\begin{equation}\label{eq: iteration inexact eigenvector}
\begin{aligned}
\|\zeta_{n+1}\|&=\|F_2(\zeta_n)\|\leq \|F(\zeta_n)\|+\|F(\zeta_n)-F_2(\zeta_n)\|\\
&\leq (1-\Delta t(\mu-2\theta L))\|\zeta_n\|+C_{3,r}\|\zeta_n\|^2.
\end{aligned}
\end{equation}
Same as in Proposition \ref{convergence analysis normal case}, the iteration \eqref{eq: iteration inexact eigenvector} yields a linear convergence rate of $1-\frac{\mu}{2L}+\theta$,
which depends on both the condition number and the error of the approximate eigendirection.
\end{proof}

\subsection{Discrete momentum-based constraint saddle dynamics}\label{subsec: momentum with exact eigenvector}
We analyze the convergence properties and rates of discrete MCSD in this subsection, showing that momentum reduces the condition-number dependence of the convergence rate from $\kappa$ to $\sqrt{\kappa}$.

By discretizing the continuous-time MCSD in \eqref{dynamics with momentum with Euclidean Language} with $\gamma=1-\gamma_1\Delta t, \gamma_2=1$, we obtain a discrete-time algorithm with exact eigenvector assumption: 
\begin{equation}\label{equation: iteration in tangential space with momentum}
\begin{cases}
    \bar{\rvec}_n=-\Delta t \left(\I-2\sum_{i=1}^k {\vvec}_{i,n}{\vvec}_{i,n}^\top\right)\gradient f(\x_n)+\gamma\rvec_n,\\
    \x_{n+1}=\mathrm{Ret}_{ \x_n}(\bar{\rvec}_n),\\
    \rvec_{n+1}=P_{\x_n\rightarrow\x_{n+1}}   \bar{\rvec}_n,\\

    \{\vvec_{i,n+1}\}_{i=1}^k=\text{EigenSolver}(\{\vvec^i_n\}_{i=1}^k,\hessian f(\x_{n+1})),
\end{cases}
\end{equation}
with $\x_0\in \M,\rvec_0=\mathbf{0},\{\vvec_i\}_1^k \in T_{\x_0}$. 
\begin{remark}
    In the continuous-time MCSD in \eqref{dynamics with momentum with Euclidean Language}, the parameter $\gamma_1$ acts as a damping coefficient that controls the rate of momentum dissipation, and thus, a smaller $\gamma_1$ (equivalently, a larger $\gamma$) results in stronger momentum effects. The parameter $\gamma_2$ acts as a relaxation parameter that determines the time scale of the dynamics. For simplicity, we set $\gamma_2=1$, corresponding to the same time scale for the dynamics of $\x$ and $\rvec$. The results in this subsection can be directly extended to general $\gamma_2>0$.
\end{remark}

For brevity, we assume that the retraction is given by the exponential map \cite{boumal2023introduction,hu2020brief}. The results, however, remain valid for more general retractions, since their updates coincide at first order and differ only in higher-order terms that do not affect the analysis.

The following lemma converts the iteration on the manifold into an iteration in the tangent space $T_{\x^*}\M$, which is convenient for our convergence analysis.
\begin{lemma}[Iteration expressed in local coordinates]\label{lemma: iteration in tangential space}
Suppose the Assumption \ref{assumption: Hessian} holds. There exists an $r>0$ such that if ${\x_n}\in \mathcal{M}$ satisfies \(d(\x_n,\x^*)<r\), we have the well-defined local coordinates
$$
\zeta_n := \exp_{\x^*}^{-1}(\x_n)\in T_{\x^*}\M,
\rho_n := P_{x_n\to \x^*} \rvec_n\in T_{\x^*}\M.
$$
Moreover, there exists an $r^*>0$, if $\left\|\begin{pmatrix}
\zeta_{n}\\
\rho_{n}
\end{pmatrix}\right\|=\sqrt{\|\zeta_{n}\|^2+\|\rho_{n}\|^2}<r^*$, the iteration in \eqref{equation: iteration in tangential space with momentum} can be written in the form of the local coordinates as
\begin{equation}\label{eq: iteration in the tangential plane with second order error}
\begin{pmatrix}
\zeta_{n+1}\\
\rho_{n+1}
\end{pmatrix}=
\underbrace{\begin{pmatrix}
\I - \Delta t R^* H^* & \gamma \I\\
-\Delta t R^* H^* & \gamma \I
\end{pmatrix}}_{\displaystyle := A}
\begin{pmatrix}
\zeta_n\\
\rho_n
\end{pmatrix}
+res_n,
\end{equation}
where $A: T_{\x^*}\M\oplus T_{\x^*}\M\rightarrow T_{\x^*}\M\oplus T_{\x^*}\M$ is defined as a linear operator, 
$R^* := \I - 2\sum_{i=1}^k \vvec_i^*{\vvec_i^*}^\top$,
$H^*:= \hessian f(\x^*),
$ and the remainder satisfies
  $
  \|res_n\|_2\leq  C\big(\|\zeta_n\|^2+\|\rho_n\|^2\big)
  $
  for some constant $C>0$. Here, \(\|\cdot\|\) denotes the Euclidean norm induced by the Riemannian metric at \(\x^*\), i.e. the $L_2$-norm on \(T_{\x^*}\M\).
\end{lemma}
\begin{proof}
From Lemma \ref{local diffeomorphism}, there exists exists an $r>0$ such that if \({\x_n}\subset \mathcal{M}\) satisfies \(d(\x_n,\x^*)\leq r\), the local coordinates exist. Let
$$
\x(\zeta):=\exp_{\x^*}(\zeta),
\widehat{\grad f}(\zeta):=P_{\x(\zeta)\to x^*}\big(\grad f(\x(\zeta))\big),
$$
and define similarly
$$
\rho_n = P_{\x_n\to x^*}\rvec_n,
\bar\rho_n = P_{\x_n\to \x^*}\bar \rvec_n.
$$
By standard normal-coordinate expansions of the Riemannian gradient, we have:
$$
\widehat{\grad f}(\zeta)
= H^* \zeta + g_2(\zeta),
\|g_2(\zeta)\|\leq C_1 \|\zeta\|^2.
$$
Let
$$
\widehat{R}(\zeta) :=
P_{\x(\zeta)\to \x^*}\circ R(\x(\zeta))\circ P_{\x^*\to \x(\zeta)}.
$$
Since \(R(\x)\) is smooth from Lemma \ref{lem:Riesz-proj-smooth}, and parallel transport is also smooth, we have that
$$
\widehat{R}(\zeta) = R^* + R_1(\zeta),
\|R_1(\zeta)\|_2\leq C_2\|\zeta\|.
$$
Thus the momentum update satisfies:
\begin{equation}\label{eq: barrho and rho}
\begin{aligned}
\bar\rho_n
= P_{\x_n\to \x^*}\bar \rvec_n
&= -\Delta t \widehat{R}(\zeta_n)\widehat{\grad f}(\zeta_n)+ \gamma\rho_n\\
  &= -\Delta t R^* H^* \zeta_n + \gamma\rho_n + e_n,
  \end{aligned}
\end{equation}
  where
$
  \|e_n\|
  \le C_3\|\zeta_n\|^2 .
$
Define the local coordinate update map
$$
F(\zeta,\eta)
:= \exp_{\x^*}^{-1}\left(
\exp_{\x(\zeta)}\big(P_{\x^*\to \x(\zeta)}\eta\big)
\right).
$$
Then, we have that \(\x_{n+1}=\x(\zeta_{n+1})\) with
$
\zeta_{n+1}=F(\zeta_n,\bar\rho_n).
$
Since 
$$
F(0,0)=0,
F(\zeta,0)=\zeta,
F(0,\eta)=\eta,
$$
the first-order derivatives satisfy:
$$
\partial_\zeta F(0,0)=\I,
\partial_\eta F(0,0)=\I.
$$
Hence, the second-order Taylor expansion yields
\begin{equation}\label{eq: taylor expansion of iteration of the momentum algorithm}
F(\zeta,\eta)
= \zeta + \eta + q_F(\zeta,\eta),
\|q_F(\zeta,\eta)\|\leq C_4(\|\zeta\|^2+\|\eta\|^2).
\end{equation}
Substituting \((\zeta,\eta)=(\zeta_n,\bar\rho_n)\) into \eqref{eq: taylor expansion of iteration of the momentum algorithm} gives
\begin{equation}\label{eq: zeta update}
\zeta_{n+1}=F(\zeta_n,\bar\rho_n)
=\zeta_n + \bar\rho_n + q_n,
\|q_n\|\leq C_4(\|\zeta_n\|^2+\|\bar\rho_n\|^2).
\end{equation}
Substituting \eqref{eq: barrho and rho} into \eqref{eq: zeta update} yields
\begin{equation}\label{eq: zeta linear}
\zeta_{n+1}
=(\I - \Delta t R^*H^*)\zeta_n + \gamma\rho_n + res_n^{(1)},
\|res_n^{(1)}\|\leq C'_4(\|\zeta_n\|^2+\|\rho_n\|^2).
\end{equation}
Note that $\x(\zeta_n)=\x_n, \x(F(\zeta_n,\bar \rho_n))=\x(\zeta_{n+1})=\x_{n+1}$. The momentum update is
\begin{equation}\label{eq: momentum update}
\rho_{n+1}
= P_{\x_{n+1}\to \x^*} r_{n+1}
= G(\zeta_n,\bar\rho_n),
\end{equation}
where
$$
G(\zeta,\eta)
:=P_{\x(F(\zeta,\eta))\to \x^*}
P_{\x(\zeta)\to \x(F(\zeta,\eta))}
P_{\x^*\to \x(\zeta)}\eta.
$$
Observe that
$$
G(0,0)=0,
G(0,\eta)=\eta,
G(\zeta,0)=0,
$$
which implies
$$
\partial_\eta G(0,0)=\I, \partial_\zeta G(0,0)=0.
$$
Thus, the second-order Taylor expansion gives:
\begin{equation}\label{eq: G}
G(\zeta,\eta)=\eta + q_G(\zeta,\eta),
\|q_G(\zeta,\eta)\|\leq C_5(\|\zeta\|^2+\|\eta\|^2).
\end{equation}
Substituting \((\zeta,\eta)=(\zeta_n,\bar\rho_n)\) into \eqref{eq: G} and combining \eqref{eq: momentum update} and \eqref{eq: barrho and rho} yield
\begin{equation}\label{eq: rho linear}
\rho_{n+1}
= \bar\rho_n + r_n^{(2)}
= -\Delta t R^*H^*\zeta_n + \gamma\rho_n + res_n^{(2)},
\end{equation}
where
$$
\|res_n^{(2)}\|
\leq C'_5(\|\zeta_n\|^2+\|\rho_n\|^2).
$$
Combining \eqref{eq: zeta linear} and \eqref{eq: rho linear}, we have \eqref{eq: iteration in the tangential plane with second order error} with
$\|res_n\|_2\leq C(\|\zeta_n\|^2+\|\rho_n\|^2)$ for
$$
A :=
\begin{pmatrix}
\I - \Delta t R^* H^* & \gamma \I\\
-\Delta t R^* H^* & \gamma \I
\end{pmatrix} \text{ and }\; res_n:=\begin{pmatrix}res_n^{(1)}\\ res_n^{(2)}\end{pmatrix}.
$$
\end{proof}

\begin{lemma}[Spectrum radius of the linear iteration operator]\label{thm:HB-opt}
Suppose Assumption \ref{assumption: Hessian} holds. The spectrum of the self-adjoint operator $K:=R^*H^*$ satisfies $\sigma(K)\subset[\mu,L]$. The linear operator $A:T_{\x^*}\M \oplus T_{\x^*}\M \to T_{\x^*}\M\oplus T_{\x^*}\M$ is defined equivalently by
\begin{equation}\label{eq:A-operator}
A
\begin{pmatrix}
\zeta\\ \rho
\end{pmatrix}
=
\begin{pmatrix}
(I-\Delta t K)\zeta + \gamma \rho\\[2pt]
-\Delta t K \zeta + \gamma \rho
\end{pmatrix}.
\end{equation}
Then, we have the following optimal bound on the spectral radius
\begin{equation}\label{eq:HB-opt-rate}
\rho(A) \leq \frac{\sqrt{\kappa}-1}{\sqrt{\kappa}+1}, \kappa:=\frac{L}{\mu},
\end{equation}
by choosing 
\begin{equation}\label{eq:HB-opt-params}
\Delta t=\Delta t^* = \frac{4}{\big(\sqrt{L}+\sqrt{\mu}\big)^2}, \gamma=\gamma^* = \left(\frac{\sqrt{L}-\sqrt{\mu}}{\sqrt{L}+\sqrt{\mu}}\right)^2.
\end{equation}
\end{lemma}

\begin{proof}
From Assumption \ref{assumption: Hessian}, $H^*:\mathcal{H}\to\mathcal{H}$ is a self-adjoint operator with the spectrum given by
\[
\sigma(H^*)\subset[-L,-\mu]\cup[\mu,L], 0<\mu\leq L<\infty.
\]
$R^* := \I - 2\sum_{i=1}^k \vvec_i^*{\vvec_i^*}^\top$ is simultaneously diagonalizable with $H^*$, then $K$ is self-adjoint and admits the eigenpairs $\{(|\lambda_i^*|,\vvec_i^*)\}_{i=1}^{d-m}$, so that the spectrum of $K$ satisfies $\sigma(K)\subset[\mu,L]$. For a fixed eigenpair $(|\lambda_i^*|,\vvec_i^*)$ with $|\lambda_i^*|\in[\mu,L]$, consider the two-dimensional subspace
\[
E_{\vvec_i^*} := \operatorname{span}\{(\vvec_i^*,0),(0,\vvec_i^*)\}\subset T_{\x^*}\M\oplus T_{\x^*}\M,
\]
we have \(
T_{\x^*}\M\oplus T_{\x^*}\M
=
\bigoplus_{i=1}^{d-m} E_{\vvec_i^*}.
\)
For any $(\zeta,\rho)=(a \vvec_i^*,b \vvec_i^*)\in E_{\vvec_i^*}$, by the definition \eqref{eq:A-operator}, we have
\[
A
\begin{pmatrix}
\zeta\\ \rho
\end{pmatrix}
=
\begin{pmatrix}
( I-\Delta t K)(a \vvec_i^*)+\gamma(b \vvec_i^*)\\
-\Delta t K(a \vvec_i^*)+\gamma(b \vvec_i^*)
\end{pmatrix}
=
\begin{pmatrix}
(1-\Delta t|\lambda_i^*|)a \vvec_i^* + \gamma b \vvec_i^*\\
-\Delta t |\lambda_i^*| a \vvec_i^* + \gamma b \vvec_i^*
\end{pmatrix}.
\]
Thus, in the basis of $E_{\vvec_i^*}$, the restriction $A|_{E_{\vvec_i^*}}$ is represented by the $2\times 2$ matrix
\begin{equation*}
B_i=
\begin{pmatrix}
1-\Delta t|\lambda_i^*| & \gamma\\
-\Delta t |\lambda_i^*| & \gamma
\end{pmatrix}.
\end{equation*}
Hence,
$\rho(A) = \max_{1\leq i\leq d-m} \rho(B_i)$.
The characteristic polynomial of $B_i$ is
\[
p(\lambda)
= \det\big(\lambda I - B_i\big)
= \lambda^2 - (1+\gamma-\Delta t|\lambda_i^*|)\lambda + \gamma.
\]
Consequently, the eigenvalues of $B_i$ are
\begin{equation}\label{eq:lambda-pm}
\lambda_\pm
=
\frac{1}{2}\Big[
1+\gamma-\Delta t|\lambda_i^*|
\ \pm\
\sqrt{(1+\gamma-\Delta t|\lambda_i^*|)^2 - 4\gamma}
\Big].
\end{equation}
Note that the product is $\lambda_+\lambda_- = \gamma>0$.
By choosing
$\Delta t = 
\Delta t^*$ and $\gamma=\gamma^*$ given by \eqref{eq:HB-opt-params},
we have 
$$
(1+\gamma^*-\Delta t^*|\lambda_i^*|)^2 - 4\gamma^*=\frac{4((L+\mu-2|\lambda_i^*|)^2-(L-\mu)^2)}{(\sqrt{L}+\sqrt{\mu})^4}\leq 0.$$
Then the two eigenvalues, $\lambda_\pm$, have the same real part with
\[|\lambda_+| = |\lambda_-| =\sqrt{\gamma^*}=\frac{\sqrt{\kappa}-1}{\sqrt{\kappa}+1}.
\]
We thus have \eqref{eq:HB-opt-rate}.
\end{proof}

\begin{lemma}\label{L2 norm of the iteration matrix}
Under the assumptions of the  Lemma \ref{thm:HB-opt}, by choosing $\Delta t = 
\Delta t^*$ and $\gamma=\gamma^*$ given by \eqref{eq:HB-opt-params},
there exists a constant $C_0\geq 1$ such that
$$
\|A^k\|_2:=\max_{\|\vvec\|=1,\vvec \in T_{\x^*}\oplus T_{\x^*}} \|A^k \vvec\|\leq C_0 \left(1-\frac{1}{1+\sqrt{\kappa}}\right)^k, k=1,2,\cdots.
$$
\end{lemma}
\begin{proof}
    From the Gelfand fomula and Lemma \ref{thm:HB-opt}, we have that $$\rho(A)=\lim_{k\rightarrow \infty}\|A^k\|_2^{\frac{1}{k}}\leq 1-\frac{2}{1+\sqrt{\kappa}}.$$
    Thus, 
    there exists a large enough $k_0$ such that if $k>k_0$, 
    $$\|A^k\|_2\leq \left(1-\frac{1}{1+\sqrt{\kappa}}\right)^k\leq C_0 \left(1-\frac{1}{1+\sqrt{\kappa}}\right)^k.$$     
    For $k\leq k_0$, choosing $$C_0=\max\left(1,\max_{1\leq k\leq k_0}\|A^k\|_2 \left(1-\frac{1}{1+\sqrt{\kappa}}\right)^{-k}\right)$$
    completes the proof.
\end{proof}

\begin{proposition}[Local convergence of the discrete MCSD]\label{linear convergence of the discretized CHiSD with momentum}
    Suppose Assumption \ref{assumption: Hessian} holds. By choosing $\Delta t = 
    \Delta t^*$ and $\gamma=\gamma^*$ given by \eqref{eq:HB-opt-params}, the discrete MCSD, i.e., the iteration point in \eqref{equation: iteration in tangential space with momentum} converges to $\x^*$ with a linear convergence rate, which is dependent on $\sqrt{\kappa}$. In other words, we first fix a constant $C_h>2C_0$ (where $C_0\geq 1$ is defined in Lemma \ref{L2 norm of the iteration matrix}), if the initial condition satisfies $
\|\zeta_{0}\|_2=r_0\leq \min\left(\frac{\sqrt{\kappa}}{(1+\sqrt{\kappa})^2CC_h^2},\frac{r^*}{C_h}\right)$ (where $r^*$ and $C$ are defined in Lemma \ref{lemma: iteration in tangential space}), we have
    \begin{equation}\label{eq: convergence rate momentum}
        \|\zeta^n\|\leq \left\|\begin{pmatrix}
\zeta_{n}\\
\rho_{n}
\end{pmatrix}\right\|_2\leq C_h \left(1-\frac{1}{1+\sqrt{\kappa}}\right)^n \left\|\begin{pmatrix}
\zeta_{0}\\
\rho_{0}
\end{pmatrix}\right\|_2.
    \end{equation}
\end{proposition}

\begin{proof}
    We prove \eqref{eq: convergence rate momentum} by induction. For $n=0$, we can directly get \eqref{eq: convergence rate momentum} by $C_h\geq 1$. Suppose \eqref{eq: convergence rate momentum} holds with $1\leq n\leq k$, and thus $\left\|\begin{pmatrix}
\zeta_{n}\\
\rho_{n}
\end{pmatrix}\right\|_2\leq C_hr_0\leq r^*, 1\leq n\leq k$. For $n=k+1$, from Lemma \ref{lemma: iteration in tangential space}, we have that 
$$
\begin{pmatrix}
\zeta_{k+1}\\
\rho_{k+1}
\end{pmatrix}
=
A^{k+1}\begin{pmatrix}
\zeta_0\\
\rho_0
\end{pmatrix}
+\sum_{i=0}^k A^{k-i} res_i.
$$
By Lemma \ref{L2 norm of the iteration matrix}, we have that 
$$
\begin{aligned}
\left\|\begin{pmatrix}
\zeta_{k+1}\\
\rho_{k+1}
\end{pmatrix}\right\|_2 &\leq C_0\left(1-\frac{1}{1+\sqrt{\kappa}}\right)^{k+1}\left\|\begin{pmatrix}
\zeta_{0}\\
\rho_{0}
\end{pmatrix}\right\|_2+C_0 C\sum_{i=0}^k \left(1-\frac{1}{1+\sqrt{\kappa}}\right)^{k-i} \left\|\begin{pmatrix}
\zeta_{i}\\
\rho_{i}
\end{pmatrix}\right\|_2^2\\
&\leq \left[  C_0\left(1-\frac{1}{1+\sqrt{\kappa}}\right)^{k+1}+C_0C_h^2Cr_0\sum_{i=0}^k \left(1-\frac{1}{1+\sqrt{\kappa}}\right)^{k+i}\right]\left\|\begin{pmatrix}
\zeta_{0}\\
\rho_{0}
\end{pmatrix}\right\|_2\\
& \leq \left(  C_0+\frac{C_0C_h^2Cr_0}{1-\frac{1}{1+\sqrt{\kappa}}}\sum_{i=0}^k \left(1-\frac{1}{1+\sqrt{\kappa}}\right)^{i}\right)\left(1-\frac{1}{1+\sqrt{\kappa}}\right)^{k+1}\left\|\begin{pmatrix}
\zeta_{0}\\
\rho_{0}
\end{pmatrix}\right\|_2\\
& \leq \left(  C_0+\frac{C_0C_h^2Cr_0}{1-\frac{1}{1+\sqrt{\kappa}}}\sum_{i=0}^{\infty} \left(1-\frac{1}{1+\sqrt{\kappa}}\right)^{i}\right)\left(1-\frac{1}{1+\sqrt{\kappa}}\right)^{k+1}\left\|\begin{pmatrix}
\zeta_{0}\\
\rho_{0}
\end{pmatrix}\right\|_2\\
&= \left(  C_0+\frac{C_0C_h^2Cr_0(1+\sqrt{\kappa})^2}{\sqrt{\kappa}}\right)\left(1-\frac{1}{1+\sqrt{\kappa}}\right)^{k+1}\left\|\begin{pmatrix}
\zeta_{0}\\
\rho_{0}
\end{pmatrix}\right\|_2\\
&\leq C_h\left(1-\frac{1}{1+\sqrt{\kappa}}\right)^{k+1}\left\|\begin{pmatrix}
\zeta_{0}\\
\rho_{0}
\end{pmatrix}\right\|_2,
\end{aligned}
$$
which completes the proof.
\end{proof}
\begin{remark}
    Compared with the linear convergence rate $1-\frac{1}{2\kappa}$ in Proposition \ref{convergence analysis normal case} without momentum acceleration, the use of momentum improves the convergence rate to $1-\frac{1}{1+\sqrt{\kappa}}$. This improvement is particularly significant when the problem is ill-conditioned, that is, when the condition number $\kappa$ is very large, which is common in practice.
\end{remark}

\section{Local convergence without exact-eigenvector assumption}\label{sec: one step v update}
In Section \ref{sec: CSD with exact eigenvector}, as well as in the analysis of SD in \cite{luo2022sinum,su2025improved,luo2025accelerated}, the convergence analysis is carried out under the assumption that the unstable eigendirections are computed exactly at each iterate. The latter could be associated with a high computational cost that is often unnecessary in practice. Motivated by this, we study the local convergence of the explicit Euler discretizations of SD, CSD, and MCSD in this section, and show that a one-step eigenvector update is sufficient to guarantee the local convergence and does not affect the tail convergence rate. We first consider the Euclidean setting (SD) in Subsection \ref{sec: one step inner in the Euclidean case}, and then extend the convergence results to the manifold setting (CSD and MCSD) in Subsection \ref{sec: one inner step on the manifold}.
\subsection{Local convergence of the Euler discretizations of SD}\label{sec: one step inner in the Euclidean case}
The SD for searching an index-$k$ saddle point $\x^*$ in the Euclidean case \cite{2019High} is defined by 
$$
    \left\{
    \begin{aligned}
		\dot{\x}&=- \left(\I-2\sum_{i=1}^k {\vvec}_i{\vvec}_i^\top\right)\nabla f(\x), \\
		\dot{\vvec}_i&=-   \left(\I-{\vvec}_i{\vvec}_i^\top-\sum_{j=1}^{i-1}2{\vvec}_j{\vvec}_j^\top\right)\nabla^2 f(\x) \vvec_i,\ i=1,\cdots,k ,\\
    \end{aligned}
    \right.
$$
where $1\leq k\leq d-1$ and $\I$ is the identity operator, which can also be written as 
$$
    \left\{
    \begin{aligned}
		\dot{\x}&=- \left(\I-2VV^\top\right)\nabla f(\x), \\
		\dot{V}&=-   \left(\I-VV^\top\right)\nabla^2 f(\x) V.\\
    \end{aligned}
    \right.
$$
where $V=[\vvec_1,\cdots,\vvec_k]$. By defining $R:=\I-2VV^\top$, we have a more compact form \cite{su2025improved}
\begin{equation}\label{eq: compact saddle dynamics}
    \left\{
    \begin{aligned}
		\dot{\x}&=-R\nabla f(\x), \\
		\dot{R}&=\nabla^2f(\x)-R\nabla^2f(\x)R.\\
    \end{aligned}
    \right.
\end{equation}
The first-order explicit Euler scheme of \eqref{eq: compact saddle dynamics} is defined as 
\begin{equation}\label{eq: compact saddle algorithm}
\begin{aligned}
    &\x_{n+1}=\x_n-\Delta t R_n\nabla  f(\x_n),\\
    & \bar{R}_{n+1}=R_n+\Delta t (\nabla^2f(\x_n)-R_n\nabla^2f(\x_n)R_n) \\
    & \bar{R}_{n+1}= V \Sigma V^\top, R_{n+1}=V\text{diag}(\underbrace{-1,\cdots,-1}_k,1,\cdots,1)V^\top, 
\end{aligned}
\end{equation}
where $\Delta t$ is the step size, and the final orthogonal step is used to ensure that $R$ is a Householder transformation. Specifically, the Householder transformation $R$ resides within the manifold
\begin{equation}
    \mathcal{R}_k=\{R\in \text{Sym}_d: R^2=\I, \text{dim ker}(R+\I)=k\}
\end{equation}
which represents an immersed submanifold of the Grassmannian Grk$(\mathbb{R}_ n)$ within the space of symmetric matrices Sym$_n$ equipped with the Frobenius inner product $\langle \cdot,\cdot \rangle_F$. The tangent space is expressed as $$
T_R\mathcal{R}_k=\{A\in \text{sym}_d,AR+RA=0\},
$$
and the projection operator is defined as 
$$\Pvec_{T_R} A=\frac{1}{2}(A-RAR).$$

Analogous to Assumption \ref{assumption: Hessian}, we introduce the following strict saddle point assumption in the Euclidean setting.
\begin{assumption}\label{assumption: Hessian one inner step}
There exist $\delta>0, L\geq\mu>0$, if $\x\in U_\delta=\{\x\in \mathbb{R}^d, \|\x-\x^*\|_2\leq \delta\}$, then the eigenvalues of $\nabla^2 f(\x)$ satisfy
$$
    -L\leq\lambda_1\leq \cdots\leq\lambda_k\leq -\mu < 0< \mu\leq\lambda_{k+1} \leq \cdots \leq \lambda_d \leq L.
$$
\end{assumption}
We first demonstrate that the orthogonal step in \eqref{eq: compact saddle algorithm} is actually a nearest point mapping from the Euclidean space to $\mathcal{R}_k$, which is specified in the following lemma.

\begin{lemma}\label{lemma: nearest point projection}
Suppose that Assumption \ref{assumption: Hessian one inner step} holds. Define the nearest point projection
\begin{equation}\label{eq: nearest point mapping}
\operatorname{Ret}(\bar R)
:= \arg\min_{R\in \mathcal{R}_k}\|R-\bar R\|_F^2.
\end{equation}
If $\Delta t < \frac{1}{2L}$, we have that $R_{n+1} = \operatorname{Ret}(\bar R_{n+1})$, and $R_{n+1}$ is the unique
minimizer of the above problem, where $R_n$, $\bar R_{n+1}$, and $R_{n+1}$ are defined in \eqref{eq: compact saddle algorithm}. Consequently, the third step of
\eqref{eq: compact saddle algorithm} defines a retraction, which satisfies
$$D_R\operatorname{Ret}(R)|_{R=\bar R}[E]: = \frac{d}{dt}\operatorname{Ret}(\bar R+tE)\big|_{t=0}=\Pvec_{T_{\bar R}}E$$
for all $\bar R \in \mathcal{R}_k, E\in \mathrm{Sym}_d$.
\end{lemma}

\begin{proof}
Note that $\bar R_{n+1}\in \mathrm{Sym}_d$ because both $R_n$ and
$\nabla^2 f(\x_n) - R_n \nabla^2 f(\x_n) R_n$ are symmetric matrices. Let
$\bar R_{n+1} = V \Lambda V^\top$ be its eigendecomposition and set
$\tilde R = V^\top R V$. Then
\[
\|\bar R_{n+1} - R\|_F^2
= \|V^\top(\bar R_{n+1} - R)V\|_F^2
= \|\tilde R - \Lambda\|_F^2,
\]
and hence,
\[
\min_{R\in \mathcal{R}_k}\|R-\bar R_{n+1}\|_F^2
\;\Longleftrightarrow\;
\min_{\tilde R\in \mathcal{R}_k}\|\tilde R - \Lambda\|_F^2
\;\Longleftrightarrow\;
\max_{\tilde R\in \mathcal{R}_k} \operatorname{tr}(\tilde R \Lambda).
\]
The last equivalence follows from the facts that $\|\tilde R\|_F^2 = d$ ($\tilde R\in\mathcal{R}_k$) and that $\|\Lambda\|_F^2$ is constant.

By von Neumann's trace inequality, for any $\tilde R \in \mathcal{R}_k$,
\begin{equation}\label{Von Neumann}
\operatorname{tr}(\tilde R \Lambda)
\;\leq\;
\sum_{i=k+1}^d \Lambda_i \;-\; \sum_{i=1}^k \Lambda_i
= \operatorname{tr}\!\left(
\operatorname{diag}(\underbrace{-1,\dots,-1}_k,1,\dots,1)\,\Lambda
\right),
\end{equation}
and equality holds if and only if $\tilde R$ is diagonal with exactly $k$
entries equal to $-1$ and $d-k$ entries equal to $1$ on the diagonal, aligned
with the ordering of the eigenvalues in $\Lambda$.

If $\Delta t < \frac{1}{2L}$, by Weyl inequality, we have
$$
|\Lambda_i - \lambda_{n,i}|
\leq
\Delta t\,\bigl\|\nabla^2 f(\x_n) - R_n \nabla^2 f(\x_n) R_n\bigr\|_2
< 1, \forall i,
$$
where $\lambda_{n,i}$ is the $i$-th smallest eigenvalue of $R_n$. Hence,
$$
\Lambda_{k+1} - \Lambda_k>
\lambda_{n,k+1} - \lambda_{n,k} - 2
= 0,
$$
which means there is a positive eigengap between the $k$-th and $(k+1)$-st eigenvalues
of $\Lambda$. Therefore, by the condition of equality in
\eqref{Von Neumann}, the matrix
\[
\operatorname{diag}(\underbrace{-1,\dots,-1}_k,1,\dots,1)
\]
is the unique maximizer of $\operatorname{tr}(\tilde R \Lambda)$, or equivalently,
the unique minimizer of $$\min_{\tilde R\in \mathcal{R}_k}\|\tilde R - \Lambda\|_F^2,$$
which establishes the uniqueness of $R_{n+1} = \operatorname{Ret}(\bar R_{n+1})$.

Finally, the fact that the third step of \eqref{eq: compact saddle algorithm}
defines a retraction with differential $D_R\operatorname{Ret}(R)|_{R=\bar R}
= \Pvec_{T_{\bar R}}$ for $\bar R \in \mathcal{R}_k$ follows from the standard
properties of the nearest point projection \cite{leobacher2021existence,dudek1994nonlinear}.
\end{proof}

\begin{proposition}[Local convergence of the Euler discretizations of SD]\label{Proposition: Local convergence of the discrete saddle dynamics in the Euclidean space}
Suppose that Assumption~\ref{assumption: Hessian one inner step} holds. If
$\Delta t < \frac{1}{2L}$, the numerical scheme given by
\eqref{eq: compact saddle algorithm} is locally contractive (and hence
asymptotically linearly stable) at $(\x^*, R^*)$.
\end{proposition}

\begin{proof}
It is sufficient to show that the spectral radius of the Jacobian of the
iteration map at $(\x^*,R^*)$ is strictly less than $1$. Denote the Jacobian by
\[
J^* \;=\;
\begin{pmatrix}
    J_{\x \x}^* & J_{\x R}^*\\
    J_{R \x}^* & J_{RR}^*
\end{pmatrix},
\]
where
\[
J_{\x \x}^* = \I - \Delta t\,R^* \nabla^2 f(\x^*), 
J_{\x R}^*[E]=-\Delta tE \nabla f(\x^*)= 0.
\]
Then, $J^*$ is a block lower triangular matrix, and it suffices to prove that the spectral radii of $J_{\x\x}^*$ and $J_{RR}^*$ are strictly less than 1.

For the $\x$-block, if $\Delta t<\frac{1}{2L}$, $\rho(J_{\x\x}^*)\leq 1-\Delta t \mu<1$.

For the $R$-block, we have that
\[
\begin{aligned}
J_{RR}^*
&= D_R\Bigl(\operatorname{Ret}\bigl(R + \Delta t\,(\nabla^2 f(\x) 
     - R\nabla^2 f(\x)R)\bigr)\Bigr)\Big|_{\x = \x^*,\, R = R^*} \\
&= D_R \operatorname{Ret}(R)|_{R=R^*} \circ D_R\Bigl(R + \Delta t\bigl(\nabla^2 f(\x) 
    - R\nabla^2 f(\x)R\bigr)\Bigr)\Big|_{\x = \x^*,\,R = R^*},
\end{aligned}
\]
where $\operatorname{Ret}$ is defined in \eqref{eq: nearest point mapping}. By Lemma~\ref{lemma: nearest point projection}, we have that $D_R \operatorname{Ret}(R)|_{R=R^*} = \Pvec_{T_{R^*}}$, which is the orthogonal projection onto the tangent space $T_{R^*}\mathcal{R}_k$.
A direct differentiation of $R + \Delta t(\nabla^2 f(\x) - R\nabla^2 f(\x)R)$
with respect to $R$ yields
$$
\begin{aligned}
&D_R\Bigl(R + \Delta t(\nabla^2 f(\x) - R\nabla^2 f(\x)R)\Bigr)\big|_{\x=\x^*,R=R^*}[E]\\
&= E - \Delta t\bigl(E\nabla^2 f(\x^*)R^* + R^*\nabla^2 f(\x^*)E\bigr),
\end{aligned}
$$
so that
$$
J_{RR}^*[E]
= \Pvec_{T_{R^*}}\bigl(E - \Delta t(E\nabla^2 f(\x^*)R^* 
    + R^*\nabla^2 f(\x^*)E)\bigr).
$$
Recall that for the manifold
$$
\mathcal{R}_k := \{R \in \mathrm{Sym}_d : R^2 = I,\ \mathrm{tr}(R)=d-2k\},
$$
the tangent space and normal space at $R^*$ are
$$
T_{R^*}\mathcal{R}_k
= \{E \in \mathrm{Sym}_d : ER^* + R^*E = 0\}, N_{R^*}\mathcal{R}_k
= \{E \in \mathrm{Sym}_d : E - R^*ER^* = 0\}.
$$
We describe the eigenstructure of $J_{RR}^*$ on these two subspaces:

\noindent\textbf{(i) Normal directions.}
Let $E \in N_{R^*}\mathcal{R}_k$, i.e., $E = R^*ER^*$. Then $\Pvec_{T_{R^*}}[E] = 0$, and we claim that $\Pvec_{T_{R^*}}(E\nabla^2 f(\x^*)R^*) = 0, \Pvec_{T_{R^*}}(R^*\nabla^2 f(\x^*)E) = 0
$ also hold. Indeed, 
$$
\begin{aligned}
\Pvec_{T_{R^*}} (E\nabla^2f(\x^*)R^*)&=\frac{1}{2}(E\nabla^2f(\x^*)R^*-R^*E\nabla^2 f(\x^*)R^*R^*)\\
&=\frac{1}{2}(R^*ER^*\nabla^2f(\x^*)R^*-R^*E\nabla^2 f(\x^*))=0.
\end{aligned}
$$ 
A similar argument
gives $\Pvec_{T_{R^*}}(R^*\nabla^2 f(\x^*)E) = 0$. Therefore
\[
J_{RR}^*[E] = 0, \forall E \in N_{R^*}\mathcal{R}_k,
\]
which means $J_{RR}^*$ has $\dim( N_{R^*}\mathcal{R}_k)$ zero eigenvalues corresponding to normal directions. 

\noindent\textbf{(ii) Tangent directions.}
Now let $E \in T_{R^*}\mathcal{R}_k$, i.e., $ER^* + R^*E = 0$. Using the spectral decompositions
$$
R^* = V S V^\top, \nabla^2 f(\x^*) = V \Lambda V^\top,
$$
where $S = \operatorname{diag}(s_1,\dots,s_d)$ with
$s_i=-1, i\leq k, s_i=1, i>k$ and $\Lambda = \operatorname{diag}(\lambda_1^*,\dots,\lambda_d^*)$.
We define 
$$
E_{ij} := \vvec_i^*{\vvec_j^*}^\top + \vvec_j^*{\vvec_i^*}^\top, i\leq k, j>k.
$$
Then each $E_{ij}$ is symmetric and satisfies $E_{ij}R^* + R^*E_{ij}=0$, and thus,
$E_{ij}\in T_{R^*}\mathcal{R}_k$. Furthermore, these $E_{ij}$ span the tangent
space (recall that the dimension of the tangential space is $k(d-k)$).

A straightforward computation yields
\[
E_{ij}\nabla^2 f(\x^*)R^* + R^*\nabla^2 f(\x^*)E_{ij}
= (\lambda_i^* s_i + \lambda_j^* s_j) E_{ij}.
\]
Since $E_{ij}\in T_{R^*}$, the projection $\Pvec_{T_{R^*}}$ acts as identity on $E_{ij}$, which leads to
\[
J_{RR}^*[E_{ij}]
= (1-\Delta t (\lambda_i^* s_i+\lambda_j^* s_j))E_{ij} = (1-\Delta t (-\lambda_i^*+\lambda_j^*))E_{ij}, i\leq k, j>k.
\]
At $(\x^*,R^*)$,  $-\lambda_i^* + \lambda_j^*$ lies in
$[2\mu,2L]$ by Assumption \ref{assumption: Hessian one inner step}. If $\Delta t < \frac{1}{2L}$, we have that
\[
\bigl|1 - \Delta t(-\lambda_i^* + \lambda_j^*)\bigr|
< 1, i\leq k, j>k.
\]
Therefore, all eigenvalues of $J_{RR}^*$ corresponding to tangent directions
lie strictly inside $(-1,1)$.

Then, we conclude that all eigenvalues $J^*$ lie in $(-1,1)$ provided $\Delta t < \frac{1}{2L}$, which completes the proof.
\end{proof}

\begin{remark}
In Propositions \ref{convergence analysis normal case} and
\ref{convergence with inexact}, we show that the inexact unstable eigenvectors
introduce an error term $\theta$ that slows down the convergence rate
from $1 - \mu/(2L)$ to $1 - \mu/(2L) + \theta$. One may worry that, since the discrete $\vvec$-dynamics update $\vvec$ by only one step at each $\x_n$, the resulting $\vvec_n$ is inherently an inexact eigenvector, thus generating a nonzero $\theta$ and slowing down the convergence. However, this does not affect the tail convergence rate. As $n \to \infty$, the iterates $\{\vvec_n^i\}_{i=1}^d$ converge to the unstable eigenvectors at the target saddle point $\x^*$, which implies $\theta \to 0$. Consequently, the convergence rate asymptotically coincides with that of the exact eigenvector case. This also explains why we do not introduce momentum acceleration for $\vvec$, since it does not improve the asymptotic convergence rate of $\x$ as $\theta \rightarrow 0$ regardless of whether momentum is added to $\vvec$-dynamics.
\end{remark}

\subsection{Local convergences of the Euler discretizations of CSD and MCSD}\label{sec: one inner step on the manifold}
In this subsection, we extend the convergence result for the Euler discretization of the SD to that of the CSD and MCSD. 

The Euler discretizations of CSD are defined as
\begin{equation}\label{eq: one inner step manifold}
\begin{cases}\x_{n+1}=\mathrm{Ret}_{\x_n}\left(-\Delta t  R_n\gradient  f(\x_n)\right),\\
     \bar{R}_{n+1}=R_n+\Delta t (\hessian f(\x_n)-R_n \hessian f(\x_n) R_n),\\
    \hat{R}_{n+1}=\mathrm{Orth}_{\x_n}(\bar{R}_{n+1}),\\
    R_{n+1}=P_{\x_n\to \x_{n+1}}\circ \hat{R}_{n+1}\circ P_{\x_{n+1}\to \x_n},
\end{cases}
\end{equation}
where $\mathrm{Orth}_{\x_n}(\bar{R}_{n+1})$ is also obtained by $$\bar{R}_{n+1} = V \Sigma V^\top, 
\hat{R}_{n+1} = V \,\mathrm{diag}(\underbrace{-1,\dots,-1}_{k},1,\dots,1)\,V^\top.$$
The operator transport step $R_{n+1}=P_{\x_n\to \x_{n+1}}\circ \hat{R}_{n+1}\circ P_{\x_{n+1}\to \x_n}$ is used to ensure that $R_n$ is a linear operator defined on $T_{\x_n}\M$ with any $n$.

Since $\x_{n+1}$ depends on $(\x_n,R_n)$, we write $\x_{n+1} = F(\x_n,R_n)$.
Similarly, $R_{n+1}$ depends on $\x_n$, $\x_{n+1}=F(\x_n,R_n)$, and $R_n$, and
thus can also be viewed as a function of $(\x_n,R_n)$, denoted by
$R_{n+1} = G(\x_n,R_n)$. Therefore, the iteration can be expressed as
$$
(\x_{n+1},R_{n+1}) = (F(\x_n,R_n),\,G(\x_n,R_n)).
$$
Similarly, the explicit Euler scheme of MCSD is defined as
\begin{equation}\label{eq: Euler scheme of the sd momentum}
\begin{cases}
    \bar{\rvec}_n=-\Delta t R_n\gradient f(\x_n)+\gamma\rvec_n,\\
    \x_{n+1}=\mathrm{Ret}_{ \x_n}(\bar{\rvec}_n),\\
    \rvec_{n+1}=P_{\x_n\rightarrow\x_{n+1}}   \bar{\rvec}_n,\\
    \bar{R}_{n+1}=R_n+\Delta t (\hessian f(\x_n)-R_n \hessian f(\x_n) R_n),\\
    \hat{R}_{n+1}=\mathrm{Orth}_{\x_n}(\bar{R}_{n+1}),\\
    R_{n+1}=P_{\x_n\to \x_{n+1}}\circ \hat{R}_{n+1}\circ P_{\x_{n+1}\to \x_n}.
\end{cases}
\end{equation}

\begin{proposition}[Local convergence of the Euler discretizations of CSD and MCSD]
Suppose that Assumption \ref{assumption: Hessian} holds. If
$\Delta t < \frac{1}{2L}$, the Euler discretizations of CSD given by
\eqref{eq: one inner step manifold} are locally contractive (and hence
asymptotically linearly stable) at $(\x^*, R^*)$, the Euler discretizations of MCSD given by \eqref{eq: Euler scheme of the sd momentum} are locally contractive at $(\x^*, \mathbf{0}, R^*)$.
\end{proposition}

\begin{proof}
We first consider the local convergence of the discretizations of CSD. Compute the Jacobian of the map $(\x_{n+1},R_{n+1}) = (F(\x_n,R_n),\,G(\x_n,R_n))$ at $(\x^*,R^*)$:
$$
J^*  =
\begin{pmatrix}
    D_\x F(\x^*,R^*) & D_R F(\x^*,R^*)\\
    D_\x G(\x^*,R^*) & D_R G(\x^*,R^*)
\end{pmatrix}.
$$
Since $\mathrm{Ret}_{\x^*}(\mathbf{0})=\x^*$, $D_\x\mathrm{Ret}_{\x^*}(\mathbf{0})=\I$, and $\grad f(\x^*)=0$, it follows that
\[
D_\x F(\x^*,R^*) = \I - \Delta t\, R^* \nabla^2 f(\x^*) : T_{\x^*}\M \to T_{\x^*}\M.
\]
Since $\grad f(\x^*) = 0$, the first-order dependence of $F$ on $R$ vanishes,
and hence, $D_R F(\x^*,R^*) = \mathbf{0}$.

Next, we analyze $D_R G(\x^*,R^*)$. Note that, for the fixed $\x=\x^*$, we have
$F(\x^*,R) = \x^*$ for all $R$, so the parallel
transport of the operator reduces to the identity:
\[
P_{\x^* \to F(\x^*,R)} \circ \hat{R} \circ P_{F(\x^*,R)\to \x^*} = \hat{R}.
\]
Thus, for $\x=\x^*$,
\[
G(\x^*,R) = \mathrm{Orth}_{\x^*}\Bigl(R + \Delta t\bigl(\nabla^2 f(\x^*) 
     - R \nabla^2 f(\x^*) R\bigr)\Bigr).
\]
Define
\[
\hat{F}(R) := R + \Delta t \bigl(\nabla^2 f(\x^*) - R \nabla^2 f(\x^*) R\bigr).
\]
Then
\[
D_R G(\x^*,R^*)[E]
= D_{\hat{F}}\mathrm{Orth}_{\x^*}(\hat{F})|_{\hat F=\hat{F}(R^*)}\bigl[D_R\hat{F}(R)|_{R=R^*}[E]\bigr].
\]
A direct differentiation of $\hat{F}(R)$ yields
\[
D_R\hat{F}(R)|_{R=R^*}[E]
= E - \Delta t\bigl(E H^* R^* + R^* H^* E\bigr),
\]
where $H^* := \hessian f(\x^*)$. By Lemma \ref{lemma: nearest point projection} (identifying $T_{\x^*}\M$ with $\mathbb{R}^{d-m}$), the derivative of
$\mathrm{Orth}_{\x^*}(\hat F)$ at $\hat F=\hat F(R^*)=R^*$ is the orthogonal projection onto the tangent
space of the reflection operator manifold at $R^*$, i.e., $D_{\hat{F}}\mathrm{Orth}_{\x^*}(\hat{F})|_{\hat F=\hat{F}(R^*)}= \Pvec_{T_{R^*}}$.
Hence
$$
D_R G(\x^*,R^*)[E]
= \Pvec_{T_{R^*}}\bigl(E - \Delta t(EH^*R^*+R^*H^*E)\bigr).
$$

Then, following exactly the same steps as in Proposition \ref{Proposition: Local convergence of the discrete saddle dynamics in the Euclidean space}, all eigenvalues of the full Jacobian $J^*$ are given by the eigenvalues of $D_\x F(\x^*, R^*)$ and $D_R G(\x^*, R^*)$, and all eigenvalues lie in $(-1,1)$ provided $\Delta t < \frac{1}{2L}$. Therefore, the scheme \eqref{eq: one inner step manifold} is locally contractive at $(\x^*,R^*)$.

For the Euler scheme of MCSD, Lemma \ref{lemma: iteration in tangential space} implies that the Jacobian of \eqref{eq: Euler scheme of the sd momentum} at $(\x=\x^*,\rvec=\mathbf{0},R=R^*)$ is given by
\[
J_{momentum}^*
=
\begin{pmatrix} 
    \I - \Delta t\, R^* H^* & \gamma \I & 0\\[0.4em]
    -\Delta t\, R^* H^*     & \gamma \I & 0\\[0.4em]
    *                       & *        & D_R G(\x^*,R^*)
\end{pmatrix}.
\]

It suffices to consider the leading $2\times 2$ block, since the spectral radius of $D_R G(\x^*,R^*)$ coincides with that of the CSD. From the proof of Lemma \ref{thm:HB-opt}, the characteristic polynomial of the \(2\times 2\) block
\[
\begin{pmatrix} 
    \I - \Delta t\, R^* H^* & \gamma \I \\[0.3em]
    -\Delta t\, R^* H^*    & \gamma \I 
\end{pmatrix}
\]
associated with an eigenvalue \(\lambda_i^*\) of \(H^*\) is $\lambda^2 - \bigl(1+\gamma - \Delta t|\lambda_i^*|\bigr)\lambda + \gamma$. Applying the Jury stability criterion to the polynomial
\[
\lambda^2 + a_1 \lambda + a_0 = 0 \text{ with }
a_1 = -\bigl(1+\gamma - \Delta t|\lambda_i^*|\bigr), a_0 = \gamma,
\]
we obtain the three conditions
\begin{equation}\label{eq: Jury}
\begin{aligned}
1 + a_1 + a_0 &= \Delta t |\lambda_i^*| > 0,\\
1 - a_1 + a_0 &= 2 + 2\gamma - \Delta t |\lambda_i^*|
              \geq 2 - \Delta t L 
              > 2 - \tfrac{1}{2} 
              > 0,\\
1 - a_0       &= 1 - \gamma > 0.
\end{aligned}
\end{equation}
Hence, all eigenvalues of this \(2\times 2\) block lie strictly inside the unit disk for any \(0 < \gamma < 1\). This implies that all the absolute eigenvalues of $J_{momentum}^*$ lie within the unit disc, i.e., the Euler discretizations of MCSD are locally contractive at $(\x^*, \mathbf{0}, R^*)$.
\end{proof}

\begin{remark}
    From \eqref{eq: Jury}, as $\gamma \rightarrow 1$, one of the characteristic roots approaches the boundary of the unit circle. This could result in damped oscillations, which in turn lead to slow convergence. Therefore, an excessively large momentum parameter $\gamma$ does not necessarily improve convergence, which will be illustrated in Section \ref{sec: numerics}. 
\end{remark}

\section{Numerical experiments}\label{sec: numerics}
In this section, we present a series of numerical examples to substantiate our theoretical results, including the linear convergence rate of the discrete CSD, the dependence of the convergence rate on the condition number, and the acceleration of MCSD. For the discrete MCSD in \eqref{eq: Euler scheme of the sd momentum}, choosing $\gamma>0$ yields momentum acceleration, whereas setting $\gamma=0$ reduces the method to the discrete CSD.

\subsection{Polynomial function on sphere and cylinder}
We begin with a simple but illustrative numerical example: polynomial functions defined on the sphere and the cylinder. Consider an energy functional constrained on $\mathbb{S}^2$,
$$
\begin{aligned}
&f(x,y,z)=(x^2-1)^2+ay^2+2az^2,\\
& \text{s.t.} \quad c(x,y,z)=x^2+y^2+z^2-1=0,
\end{aligned}
$$
where $a$ is a parameter to adjust the condition number at the target saddle point. When $a=2$ ($\kappa=3$), the energy landscape is shown in Figure \ref{fig: ball}(a), with an index-1 saddle point at $\x^*=(0,1,0)$. In Figure \ref{fig: ball}(b), the yellow line shows that the iteration points converge to $\x^*$ with a linear convergence rate. As we adjust $a=2$ ($\kappa=3$) to $a=0.1$ ($\kappa=22.2698$), the convergence rate is reduced. We now fix $a=0.1$ as a parameter representing the ill-conditioned case. It can be observed from the iteration trajectory in Figure \ref{fig: ball}(c) that the discrete CSD iteration points without momentum ($\gamma=0$, black) tend to converge more rapidly along the steep directions, while moving slowly along the eigenvector corresponding to the smallest absolute eigenvalue. This is reflected in Figure \ref{fig: ball}(d) that the error decreases rapidly at the beginning of the iteration (along the steep directions), but slows down in the later stage (along the flat directions). With momentum acceleration added, the introduction of inertia along the eigenvector associated with the smallest-magnitude eigenvalue (see the light sky blue trajectory in Figure \ref{fig: ball}(c)) accelerates the convergence along the flattest direction and consequently yields a faster overall convergence rate (see Figure \ref{fig: ball}(d)).

\begin{figure}
    \centering    \includegraphics[width=0.8\linewidth]{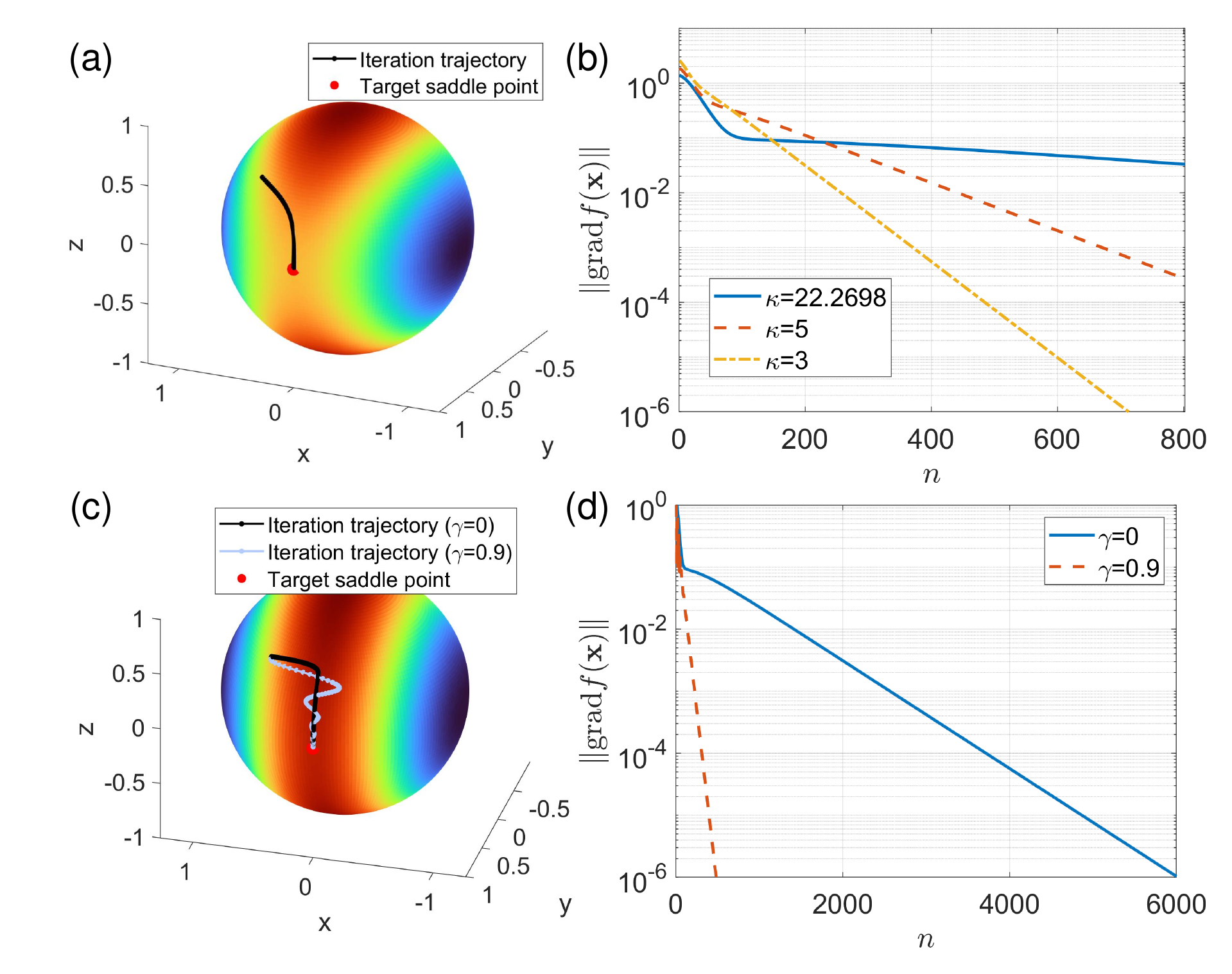}
   \vspace{-.4cm}
   \caption{(a) The energy landscape of $f(x,y,z)=(x^2-1)^2+ay^2+2az^2$ with $a=2$ defined on a unit sphere, and the trajectory of the discrete CSD without momentum. (b) Error plots for the discrete CSD without momentum under different condition numbers. As $a$ decreases, the condition number of $\hessian f(x^{*})$ increases, which in turn yields a slower linear convergence rate. (c) For (a=0.1) (ill-conditioned), trajectories of the discrete CSD without momentum ($\gamma=0$, black) and MCSD with momentum ($\gamma=0.9$, light sky blue).
(d) For $a=0.1$, error plots for the discrete CSD without momentum ($\gamma=0$) and MCSD with momentum ($\gamma=0.9$) for the step size $\Delta t=0.01$.}
    \label{fig: ball}
\end{figure}

\begin{figure}
    \centering    \includegraphics[width=0.8\linewidth]{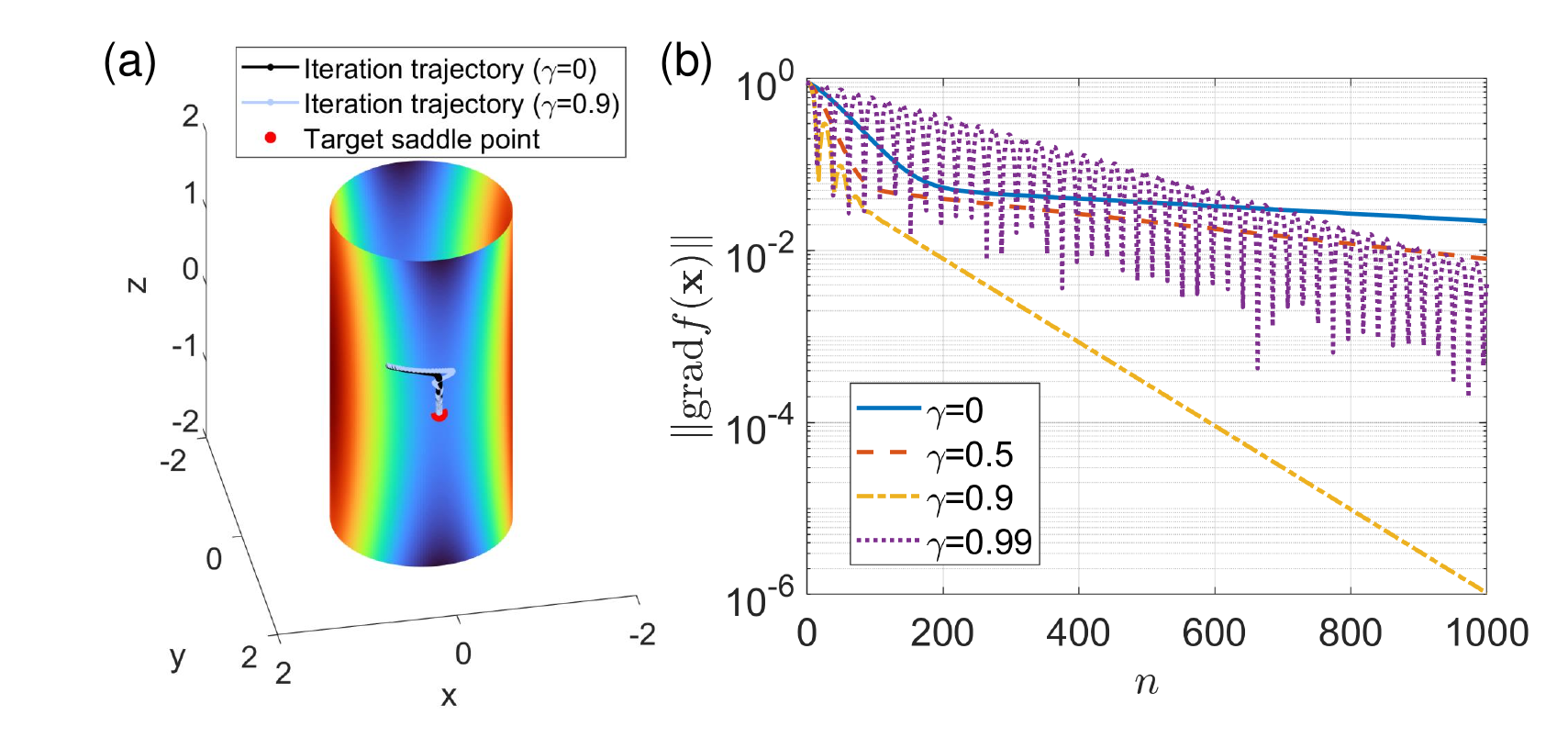}
    \vspace{-.4cm}
    \caption{(a) The energy landscape of $f(x,y,z)=-y^2-0.05z^2$ defined on a cylinder, and trajectories of the discrete CSD without momentum ($\gamma=0$, black) and MCSD with momentum ($\gamma=0.9$, light sky blue).
(d) Error plots for the discrete CSD without momentum ($\gamma=0$) and MCSD with momentum ($\gamma=0.5,0.9,0.99$) for the step size $\Delta t=0.01$.}
    \label{fig: cylinder}
\end{figure}
In many practical applications, it is common that some variables are constrained while others are unconstrained, which corresponds to a cylindrical manifold. Motivated by this, we test the performance of the discrete saddle dynamics on a polynomial defined on the cylinder,
$$
\begin{aligned}
&f(x_1,x_2,x_3)=-y^2-0.05z^2,\\
& \text{s.t.} \quad c(x,y)=x^2+y^2-1=0.
\end{aligned}
$$
In Figure \ref{fig: cylinder}(a), we plot the energy landscape of the objective function, which indicates that $\x^{*}=(0,1,0)$ is an index-1 saddle point. Since the landscape near $\x^*$ is relatively flat, although the discrete CSD still exhibits linear convergence, the convergence is slow. As we increase the momentum parameter from zero, the convergence rate initially improves, but eventually deteriorates due to the damped oscillation (see Figure \ref{fig: cylinder}(b)). While momentum accumulates velocity along the flattest direction, an excessively large momentum causes inertia to dominate the trajectory rather than the reflected gradient. Therefore, in practice, the momentum parameter should be tuned together with the step size $\Delta t$ to ensure numerical stability. Typically, a smaller step size allows for a larger range of momentum parameters.

In the next several examples, we will test the algorithms on practical engineering and scientific problems to further validate our theoretical results and demonstrate the algorithm’s applicability.
\subsection{Thomson problem}
\begin{figure}
    \centering
    \includegraphics[width=0.58\linewidth]{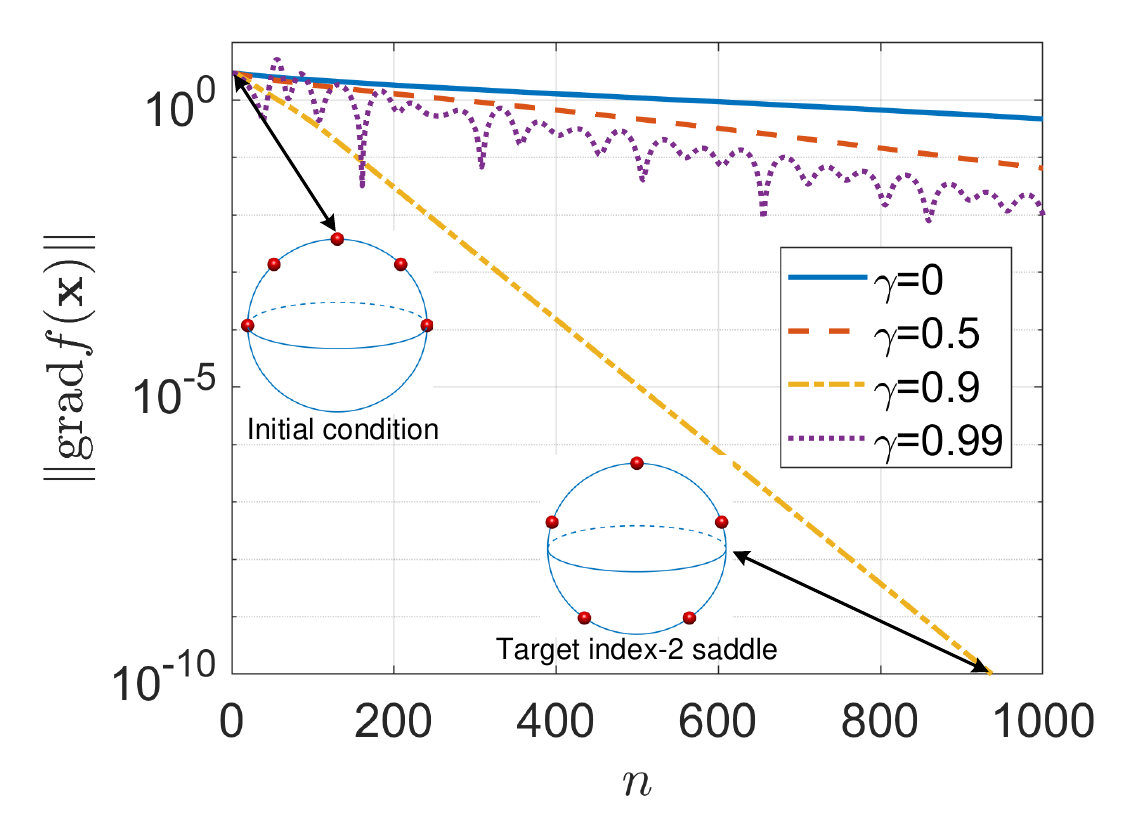}
    \vspace{-.4cm}
    \caption{Error plot for the discrete CSD without momentum ($\gamma= 0$) and MCSD with momentum ($\gamma = 0.5, 0.9, 0.99$) applied on the Thomson problem with particle number $M=5$ and step size  $\Delta t=0.001$. Each red ball represents a particle on the sphere.}
    \label{fig: Thomson}
\end{figure}

The Thomson problem concerns the arrangement of $N$ mutually repelling charges confined to the surface of a sphere. Each pair of particles interacts through the Coulomb potential $1/r$, leading to a highly non-convex energy landscape with many configurations, including minimizers and saddle points. Although it was originally proposed as a model of atomic structure, the Thomson problem also has several other concrete applications, including spherical discretization \cite{hardin2004discretizing} and the analysis of defect structures in materials science \cite{bowick2002crystalline}.

In the Thomson problem, the coordinate of the $i$-th particle $\x_i=(x_i,y_i,z_i)\in \mathbb{R}^3$ is constrained on the unit sphere, and the energy function is defined as 
$$
E(\x_1,\cdots,\x_M)= \sum_{i=1}^M \sum_{j=i+1}^M \frac{1}{\Vert \x_i-\x_j \Vert^2_2}.
$$
Earlier numerical studies of saddle points using CSD have been given in \cite{zhang2012constrained,yin2020constrained}. By rotational symmetry, the first particle is fixed at the north pole, while the second is constrained to lie in the $yz$-plane, so that the Riemannian Hessian at a stationary point generally has no zero eigenvalues \cite{yin2020constrained}. Hence, the variables are constrained on a manifold,
$$
\M=\left\{ (\x_1,\cdots,\x_M)\in \mathbb{R}^{3\times M}, \x_1=(0,0,1),x_2=0,\|\x_i\|_2=1 \right\}.
$$
The retraction operator for $\x \in \M, \mathbf{t}=(\mathbf{0},(0,\mathbf{t}_2),\mathbf{t}_3,\cdots,\mathbf{t}_M)\in T_\x$ is defined as 
$$
\text{Ret}_{\x}(\mathbf{t})=(\x_1,(x_2,\text{Ret}^{\mathbb{S}^1}_{(y_2,z_2)}(\mathbf{t}_2)),\text{Ret}_{\x_3}^{\mathbb{S}^2}(\mathbf{t}_3)),\cdots,\text{Ret}_{\x_M}^{\mathbb{S}^2}(\mathbf{t}_M))),
$$
where $\text{Ret}_{\x_i}^{\mathbb{S}^{d_0-1}}$ is the the exponential mapping 
\begin{equation}\label{eq: exponential mapping of unit ball}
\text{Exp}_\x(\mathbf{t})=(\cos\|\mathbf{t}\|_2)\x+\frac{\sin \|\mathbf{t}\|_2}{\|\mathbf{t}\|_2}\mathbf{t}.
\end{equation}
The vector transport of the vector $\hat \vvec=(\mathbf{0},(0,\mathbf{v}_2),\mathbf{v}_3,\cdots,\mathbf{v}_M)\in T_\x$ from $\x$ to $\text{Ret}_{\x}(\mathbf{t})$ is defined as

$$
P_{\x\to \text{Ret}_{\x}(\mathbf{t})} \vvec=\left(\mathbf{0},\left(0,P_{(y_2,z_2)\to \text{Ret}_{(y_2,z_2)}^{\mathbb{S}^1}(\mathbf{t}_2)}\mathbf{v}_2\right),P_{\x_3 \to \text{Ret}^{\mathbb{S}^2}_{\x_3}(\mathbf{t}_3)}\vvec_3,\cdots\right),
$$
where $P_{\x_i \to \text{Ret}^{\mathbb{S}^{d_0-1}}_{\x_i}(\mathbf{t}_i)}$ is the vector transport for the unit sphere $\mathbb{S}^{d_0-1}$ with the form of
\begin{equation}\label{eq: vector tansport on a unit ball}
P_{\x \to \text{Ret}^{\mathbb{S}^{d_0-1}}_{\x}(\mathbf{t})}\vvec=\mathbf{v}+\frac{\cos\|\mathbf{t}\|_2-1}{\|\mathbf{t}\|_2^2}\langle  \mathbf{t},\vvec\rangle \mathbf{t}-\frac{\sin \|\mathbf{t}\|_2}{\|\mathbf{t}\|_2}\langle  \mathbf{t},\vvec\rangle \x.
\end{equation}
Substituting the retraction operator and vector transport into the Euler scheme \eqref{eq: one inner step manifold} and \eqref{eq: Euler scheme of the sd momentum}, we can search for the saddle point of the Thompson problem as shown in Figure \ref{fig: Thomson}. The minimizer of the Thomson problem is the uniform distribution. The energy landscape also contains many saddle points. For instance, when $M=5$, the configuration in which the particles are uniformly spaced on the same circle is an index-2 saddle point. To locate this saddle, we initialize all particles in the upper half of the domain. Both the discrete CSD and MCSD methods exhibit linear convergence, with the discrete MCSD converging more rapidly. However, an excessively large momentum parameter (e.g., $\gamma = 0.99$) can introduce a damped oscillation that ultimately slows the convergence.

\subsection{Rayleigh quotient}
In this numerical example, we consider a more complex manifold, the Rayleigh quotient defined on the Stiefel manifold. Let $A\in\mathbb{R}^{n\times n}$ be a symmetric matrix and let $V\in\mathbb{R}^{n\times p}$ satisfy the Stiefel constraint
\[
\mathrm{St}(n,p)=\{V\in\mathbb{R}^{n\times p}:\; V^\top V=I_p\}.
\]
The objective function of the Rayleigh quotient \cite{absil2009optimization} to search for the $p$ eigenvectors is defined as
\begin{equation}\label{eq: RQ}
f(V)=-\mathrm{tr}(V^\top A V), V\in \mathrm{St}(n,p),
\end{equation}
which can be viewed as a matrix Rayleigh quotient (equivalently, the sum of Rayleigh quotients of the $p$ orthonormal columns of $V$).
Although $f(V)$ is quadratic in $V$ in the ambient space, the orthogonality constraint induces a highly nonlinear geometry, leading to a nonconvex landscape with many minimizers and saddle points. Writing the eigenvalues of $A$ as $\lambda_1> \lambda_2> \cdots > \lambda_n$ with orthonormal eigenvectors $\{q_i\}_{i=1}^n$, the global minimizers are the Stiefel points whose columns span the dominant eigenspace, i.e., $\mathrm{span}(V)=\mathrm{span}\{q_1,\dots,q_p\}$. 
More generally, any $V$ whose columns span an invariant subspace of $A$ is a critical point. For example, when $p=1$, the problem reduces to optimizing the scalar Rayleigh quotient on the unit sphere. Then every eigenvector $q_i$ is a critical point. In particular, $q_1$ is the global maximizer, and $q_n$ is the global minimizer. 
Moreover, $q_2$ is an index-$1$ saddle point, $q_3$ is an index-$2$ saddle point, and in general, $q_i$ is an index-$(i-1)$ saddle point.

Denote the ambient Euclidean inner product by $\langle U,V\rangle=\tr(U^\top V)$. The tangent space at $X\in\mathrm{St}(n,p)$ is defined as
\[
T_X\mathrm{St}(n,p)=\{Z\in\mathbb{R}^{n\times p}:\; X^\top Z+Z^\top X=0\}.
\]
The orthogonal projection from the ambient space onto $T_X\mathrm{St}(n,p)$ is
\[
\Pvec_X(Y)=Y - X\,\mathrm{sym}(X^\top Y), \mathrm{sym}(M)=\tfrac12(M+M^\top).
\]
The Euclidean gradient of $f$ is $\nabla f(X)=-2AX$. Hence, the Riemannian gradient of \eqref{eq: RQ} is the tangent projection
\[
\grad f(X)=\Pvec_X(\nabla f(X))
= -2\Big(AX - XX^\top AX\Big).
\]
Given a tangent vector $\eta\in T_X\mathrm{St}(n,p)$, the retraction operator is defined by $\mathrm{Ret}_X(\eta)=\mathrm{qf}(X+\eta)$, where $\mathrm{qf}(\cdot)$ denotes the $Q$-factor in the thin QR decomposition:
if $X+\eta = QR$ with $Q^\top Q=I_p$ and $R\in\mathbb{R}^{p\times p}$ is an upper-triangular matrix, then $\mathrm{Ret}_X(\eta)=Q\in\mathrm{St}(n,p)$. For $Y\in\mathrm{St}(n,p)$ and $\zeta\in T_X\mathrm{St}(n,p)$, the vector transport is defined as $$P_{X\to Y}(\zeta)
= \zeta - Y\,\mathrm{sym}(Y^\top \zeta)\in T_Y\mathrm{St}(n,p).$$

\begin{figure}
    \centering
    \includegraphics[width=0.55\linewidth]{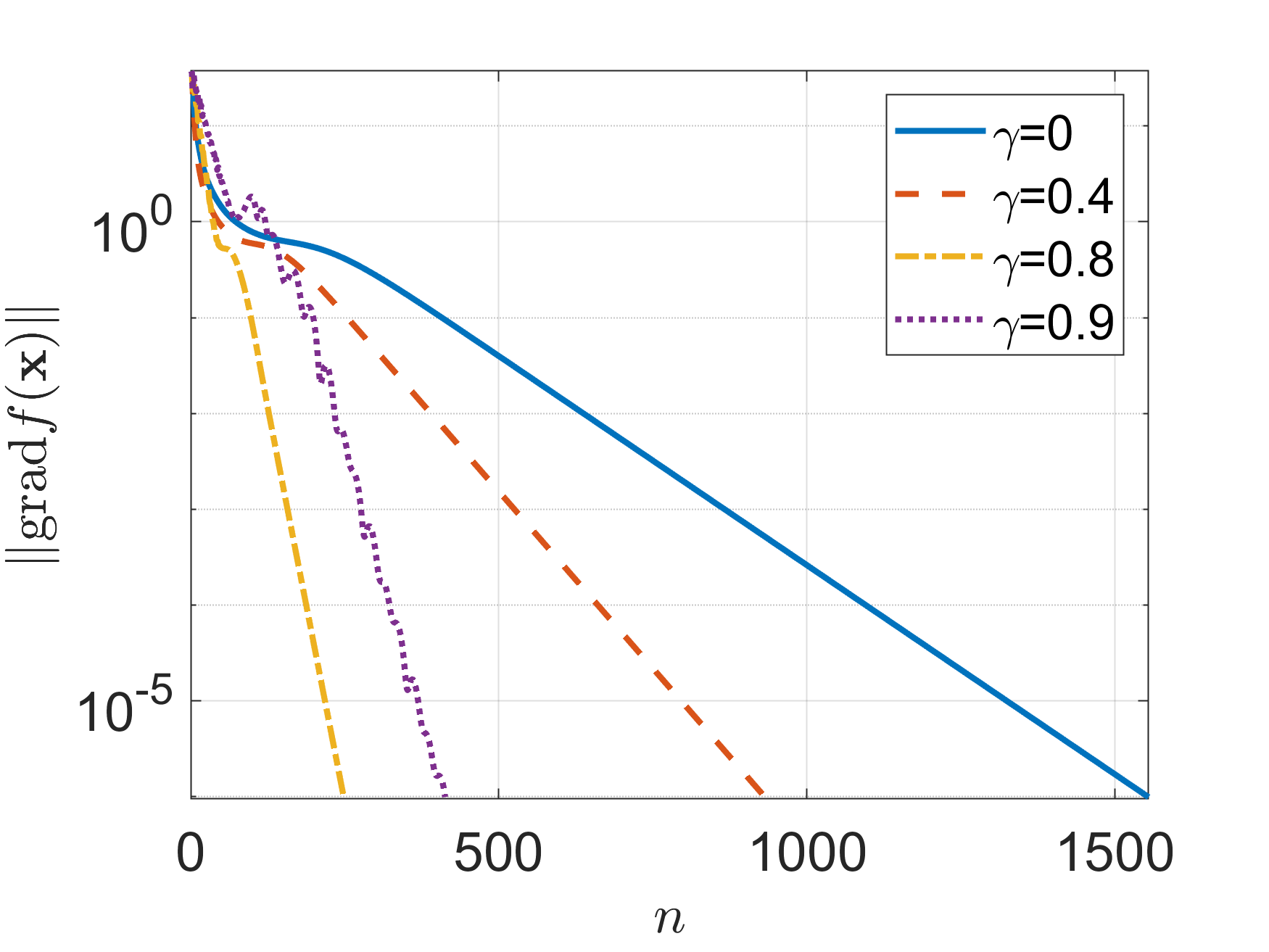}
    \vspace{-.4cm}
    \caption{Error plot for the discrete CSD ($\gamma= 0$) and MCSD ($\gamma = 0.5, 0.9, 0.99$) applied on the Rayleigh quotient problem with $n=100,p=2$. The step size is chosen as $\Delta t=0.01$.}
    \label{fig: Rayleigh}
\end{figure}

We set $n=100, p=2, k=4$ and construct a symmetric matrix $A\in\mathbb{R}^{n\times n}$ whose spectrum is prescribed as $\lambda_1=1,\lambda_2=2, \cdots, \lambda_n=n$. We generate a random orthogonal matrix $Q\in\mathbb{R}^{n\times n}$ and define $A = Q\,\mathrm{diag}(\lambda_1,\ldots,\lambda_n)\,Q^\top$ and aim to compute an index-$4$ saddle point using the discrete CSD and MCSD. In this example, the target saddle point corresponds to the Stiefel point whose column space is the invariant subspace spanned by the eigenvectors $q_2$ and $q_5$, i.e., $\mathrm{span}(V)=\mathrm{span}\{q_2,q_5\}$.
In Figure \ref{fig: Rayleigh}, starting from a fixed initial point $V_0\in\mathrm{St}(n,2)$, we apply the discrete CSD and MCSD to drive the iterate toward this index-$4$ saddle. Similarly, we observe that both the discrete CSD and MCSD exhibit a linear convergence rate. However, incorporating a moderate momentum term can significantly accelerate convergence, especially for ill-conditioned cases (or the eigen-gap of $A$ is small).

\subsection{Bose-Einstein condensate}
\begin{figure}
    \centering
    \includegraphics[width=0.7\linewidth]{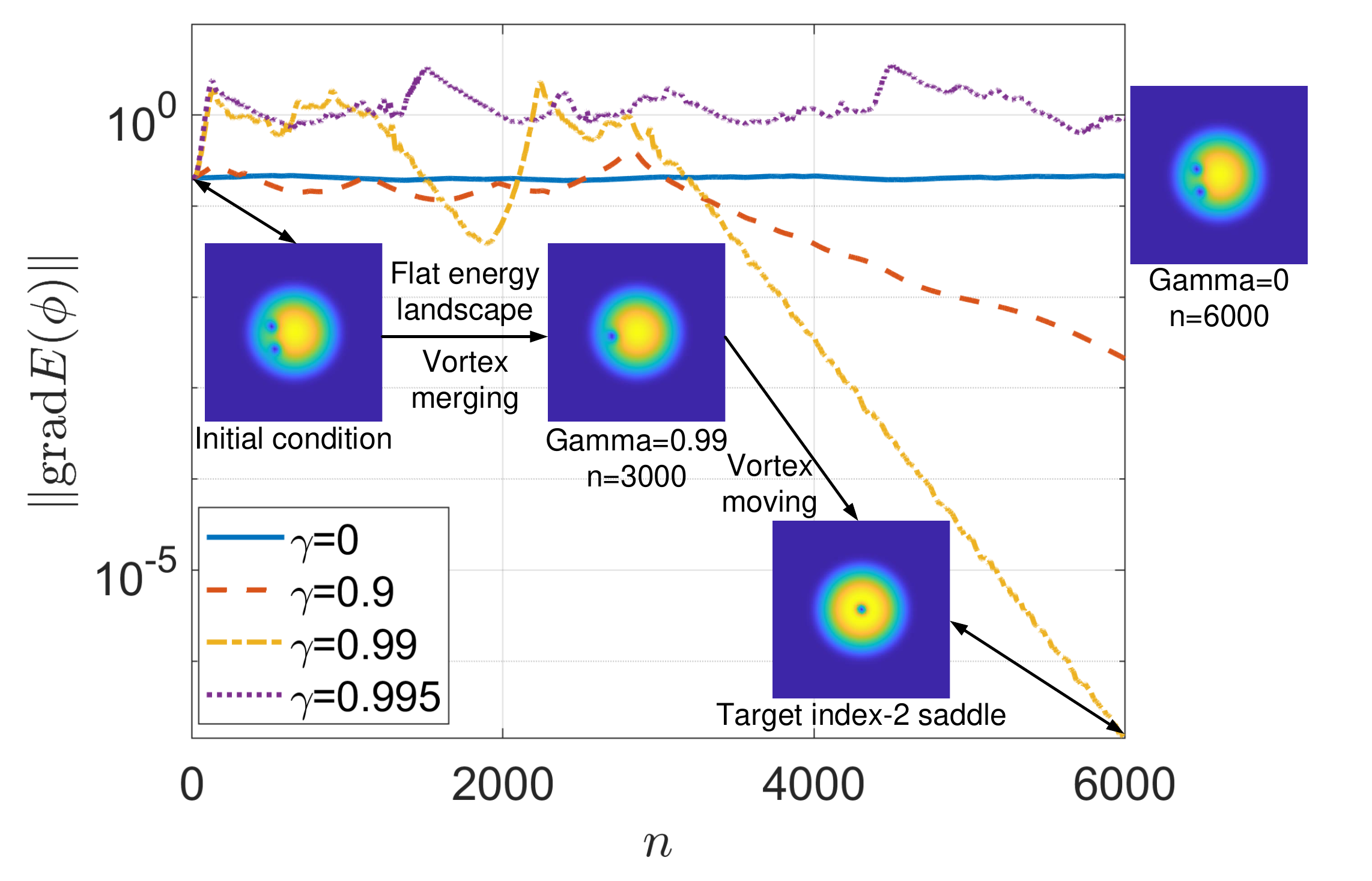}
    \vspace{-.4cm}
    \caption{Error plot for the discrete CSD without momentum ($\gamma= 0$) and MCSD with momentum ($\gamma = 0.9, 0.99, 0.995$) applied on the BEC functional. The step size is chosen as $\Delta t=0.001$. The color bar represents the probability density $|\phi|^2$ for the state. The subscript $\gamma=0,n=6000$ corresponds to the configuration obtained after $n=6000$ iterations of the discrete CSD without momentum ($\gamma=0$).
}
    \label{fig: BEC}
\end{figure}
To further validate our theoretical results and to illustrate that momentum not only accelerates tail convergence but can also help escape flat landscapes, we conduct numerical experiments on the Bose-Einstein condensate (BEC) energy functional, which is of great interest in computational physics research \cite{henning2025gross,bao2003ground,bao2004computing}. The behavior of a Bose-Einstein condensate can be characterized by a macroscopic wave function $\psi(\mathbf{x},t)$, which evolves according to the following nonlinear Gross-Pitaevskii equation
\[
i\,\partial_t \psi(\mathbf{x},t)
=
\left(
-\frac{1}{2}\Delta
+ V(\mathbf{x})
+ g\,|\psi(\mathbf{x},t)|^2
\right)\psi(\mathbf{x},t),
\]
where $V(\x)$ is a real-valued trapping potential. Since $\psi(\mathbf{x},t)$ is a wave function, it naturally satisfies the normalization constraint
$
\int_{\mathbb{R}^d} \lvert \psi(\mathbf{x},t)\rvert^2 \mathrm{d}\mathbf{x} = 1.
$
To find a stationary solution to the Gross-Pitaevskii equation, we take the ansatz $\psi(\mathbf{x},t)=\phi(\mathbf{x})e^{-i\mu t}$. Then, the solution of the nonlinear elliptic eigenvalue problem can be characterized as a critical point of the following energy functional
$$
E(\phi)=\int_{\mathbb{R}^d}\left(\frac{1}{2}\lvert\nabla \phi(\mathbf{x})\rvert^2+V(\mathbf{x})\lvert\phi(\mathbf{x})\rvert^2+\beta \lvert\phi(\mathbf{x})\rvert^4\right) \mathrm{d}\mathbf{x},
$$
subject to the normalization constraint $\int_{\mathbb{R}^d}\lvert \phi(\mathbf{x})\rvert^2\mathrm{d}\mathbf{x}=1$.

We consider a two-dimensional BEC system with a radial-symmetric harmonic oscillator $V (\x) = \|\x\|^2/2$ and $\beta = 300$. In the numerical computation, the wave function $\phi$ is truncated into a bounded domain $D = [-M, M]^2$ with the corresponding energy functional
\begin{equation}
E(\phi)=\int_{D}\left(\frac{1}{2}\lvert\nabla \phi(\mathbf{x})\rvert^2+V(\mathbf{x})\lvert\phi(\mathbf{x})\rvert^2+\beta \lvert\phi(\mathbf{x})\rvert^4\right) \mathrm{d}\mathbf{x},
\label{eq: BEC energy functional}
\end{equation}
and homogeneous Dirichlet boundary conditions, because the stationary states decay to zero exponentially fast in the far field from the effect of the trapping potential. We discretize the wave function with $M = 8$ using finite difference methods with $64$ nodes along each dimension. Since the wave function lies in the complex-valued function space, the inner product is defined as
$$
\langle \phi_1,\phi_2 \rangle=\int_D \frac{\phi_1\bar \phi_2+\phi_2\bar \phi_1}{2} \mathrm{d}\x.
$$
The Euclidean and Riemannian gradients of the BEC energy functional \eqref{eq: BEC energy functional} are 
$$
\nabla E(\phi)=-\Delta \phi+2V(\x)\phi+2\beta|\phi|^2\phi, \gradient E(\phi)=\nabla E(\phi)-\langle \nabla E(\phi),\phi \rangle \phi.
$$
The retraction operator and the corresponding vector transport for this unit sphere are chosen as \eqref{eq: exponential mapping of unit ball} and \eqref{eq: vector tansport on a unit ball}.

The global minimizer or the ground state of \eqref{eq: BEC energy functional} is the configuration without a vertex, while the configuration with a vertex of winding number $+1$ at the trap center serves as an index-2 saddle point, as shown in Figure \ref{fig: BEC}. Note that both the minimizer and the index-2 central vortex state can be written in polar coordinates as $\phi(\x) = e^{im\theta} \phi(r)$, where $m$ is the winding number. Both states possess a zero eigenvalue associated with the invariance of the energy functional \eqref{eq: BEC energy functional} under the global phase transformation $e^{i\nu}\phi(\x)$ with a real constant $\nu$. The index of the saddle only accounts for the negative eigenvalues, and not the zero eigenvalue. With the initial condition where two vertices emerge at the boundary, we search for the index-2 saddle point with the CSD and MCSD. The convergence process involves two steps: the two boundary vertices emerge into a single vertex, and the emerged vertex moves towards the center (tail convergence). Since the wave function approximates zero near the boundary, the vertex merging at the boundary does not significantly change the energy, which means that the iteration points during this process lie in a flat energy landscape. In Figure \ref{fig: BEC}, we can see that the momentum not only accelerates the tail convergence but also helps the iteration points escape from the flat energy landscape efficiently. Specifically, with 
 $\gamma=0.99$, after $3000$ iterations, the iteration points complete the vortex merging. On the contrary, with $\gamma=0$ and the same step size, after $6000$ iterations, the iteration points are still trapped in the flat energy landscape. The system is naturally ill-conditioned because the energy functional involves $|\nabla \phi|^2$. The associated Hessian operator is the negative Laplacian $-\Delta$, whose spectrum ranges from $0$ to $+\infty$. It is worth noting that such gradient terms are also common in other systems, such as phase-field models \cite{provatas2011phase} and the Oseen-Frank model for liquid crystals \cite{frank1958liquid}. In this setting, incorporating momentum may accelerate the computation of saddle points.

\section{Conclusion and discussion}\label{discussion}
This work focuses on the numerical analysis of the discrete constrained saddle dynamics and their momentum variants. Under the assumption that the exact unstable eigenvector is readily solvable, we establish local linear convergence of the discrete constrained saddle dynamics, showing that the convergence rate depends on the condition number of the Riemannian Hessian. To mitigate this dependence, we introduce a momentum-based constrained saddle dynamics and prove local convergence for both the continuous-time dynamics and the discrete numerical scheme, which demonstrates that momentum accelerates convergence, especially in ill-conditioned cases. We also prove that a single-step eigenvector update is sufficient to guarantee the local convergence; in other words, the assumption of having the exact unstable eigenvector at each iteration is not essential, which substantially reduces the computational cost of the discrete scheme. This result can be directly incorporated into existing convergence analyses for saddle dynamics \cite{luo2022sinum,su2025improved,luo2025accelerated}, which removes the assumption of exact eigenvectors made in those works, making the theory more consistent with practical implementations. Finally, we conduct several numerical experiments to substantiate our theoretical results.

This study opens several potential directions for further exploration. For instance, we use a fixed step size in the current version for simplicity, which can be replaced by an adaptive step-size strategy to further accelerate convergence \cite{hu2020brief}, and it would be worthwhile to study its convergence properties. Zeroth-order and stochastic algorithms can also be applied to constrained saddle-point computation to handle objective functions for which the Riemannian gradient is difficult to access \cite{li2023stochastic}. Finally, this work focuses primarily on local convergence, and studying global convergence will also be a great future research direction. 

\section*{Acknowledgments}
The authors would like to thank Dr. Jianyuan Yin for providing the codes for the Thompson and Bose-Einstein condensate problems.

\bibliography{references}

\providecommand{\bysame}{\leavevmode\hbox to3em{\hrulefill}\thinspace}
\providecommand{\MR}{\relax\ifhmode\unskip\space\fi MR }
\providecommand{\MRhref}[2]{%
  \href{http://www.ams.org/mathscinet-getitem?mr=#1}{#2}
}
\providecommand{\href}[2]{#2}
\begin{thebibliography}{10}

\bibitem{absil2009optimization}
P.~A. Absil, R.~Mahony, and R.~Sepulchre, \emph{Optimization algorithms on matrix manifolds}, Optimization Algorithms on Matrix Manifolds, Princeton University Press, 2009.

\bibitem{bao2004computing}
W.~Bao and Q.~Du, \emph{Computing the ground state solution of {Bose--Einstein} condensates by a normalized gradient flow}, SIAM Journal on Scientific Computing \textbf{25} (2004), no.~5, 1674--1697.

\bibitem{bao2003ground}
W.~Bao and W.~Tang, \emph{Ground-state solution of {Bose--Einstein} condensate by directly minimizing the energy functional}, Journal of Computational Physics \textbf{187} (2003), no.~1, 230--254.

\bibitem{boumal2023introduction}
N.~Boumal, \emph{An introduction to optimization on smooth manifolds}, Cambridge University Press, 2023.

\bibitem{bowick2002crystalline}
M.~Bowick, A.~Cacciuto, D.~R. Nelson, and A.~Travesset, \emph{Crystalline order on a sphere and the generalized thomson problem}, Physical Review Letters \textbf{89} (2002), no.~18, 185502.

\bibitem{dudek1994nonlinear}
E.~Dudek and K.~Holly, \emph{Nonlinear orthogonal projection}, Annales Polonici Mathematici, vol.~59, Polska Akademia Nauk. Instytut Matematyczny PAN, 1994, pp.~1--31.

\bibitem{weinan2002string}
W.~E, W.~Ren, and E.~Vanden-Eijnden, \emph{String method for the study of rare events}, Physical Review B \textbf{66} (2002), no.~5, 052301.

\bibitem{weinan2007simplified}
\bysame, \emph{Simplified and improved string method for computing the minimum energy paths in barrier-crossing events}, Journal of Chemical Physics \textbf{126} (2007), no.~16, 164103.

\bibitem{vanden2010transition}
W.~E and E.~Vanden-Eijnden, \emph{Transition-path theory and path-finding algorithms for the study of rare events}, Annual Review of Physical Chemistry \textbf{61} (2010), 391--420.

\bibitem{gad}
W.~E and X.~Zhou, \emph{The gentlest ascent dynamics}, Nonlinearity \textbf{24} (2011), no.~6, 1831.

\bibitem{frank1958liquid}
F.~C. Frank, \emph{{I. Liquid crystals.} {On} the theory of liquid crystals}, Discussions of the Faraday Society \textbf{25} (1958), 19--28.

\bibitem{goodfellow2016deep}
I.~Goodfellow, Y.~Bengio, A.~Courville, and Y.~Bengio, \emph{Deep learning}, vol.~1, MIT press Cambridge, 2016.

\bibitem{gramsbergen1986landau}
E.~F. Gramsbergen, L.~Longa, and W.~H. de~Jeu, \emph{Landau theory of the nematic-isotropic phase transition}, Physics Reports \textbf{135} (1986), no.~4, 195--257.

\bibitem{griffiths2018introduction}
D.~J. Griffiths and D.~F. Schroeter, \emph{Introduction to quantum mechanics}, Cambridge university press, 2018.

\bibitem{hardin2004discretizing}
D.~P. Hardin and E.~B. Saff, \emph{Discretizing manifolds via minimum energy points}, Notices of the AMS \textbf{51} (2004), no.~10, 1186--1194.

\bibitem{henkelman1999dimer}
G.~Henkelman and H.~J{\'o}nsson, \emph{A dimer method for finding saddle points on high dimensional potential surfaces using only first derivatives}, The Journal of chemical physics \textbf{111} (1999), no.~15, 7010--7022.

\bibitem{henning2025gross}
P.~Henning and E.~Jarlebring, \emph{The {Gross--Pitaevskii} equation and eigenvector nonlinearities: numerical methods and algorithms}, SIAM Review \textbf{67} (2025), no.~2, 256--317.

\bibitem{hu2020brief}
J.~Hu, X.~Liu, Z.~W. Wen, and Y.~X. Yuan, \emph{A brief introduction to manifold optimization}, Journal of the Operations Research Society of China \textbf{8} (2020), no.~2, 199--248.

\bibitem{leobacher2021existence}
G.~Leobacher and A.~Steinicke, \emph{Existence, uniqueness and regularity of the projection onto differentiable manifolds}, Annals of global analysis and geometry \textbf{60} (2021), no.~3, 559--587.

\bibitem{li2023stochastic}
J.~Li, K.~Balasubramanian, and S.~Ma, \emph{Stochastic zeroth-order {Riemannian} derivative estimation and optimization}, Mathematics of Operations Research \textbf{48} (2023), no.~2, 1183--1211.

\bibitem{li2001minimax}
Y.~Li and J.~Zhou, \emph{A minimax method for finding multiple critical points and its applications to semilinear pdes}, SIAM Journal on Scientific Computing \textbf{23} (2001), no.~3, 840--865.

\bibitem{lifshitz2007soft}
R.~Lifshitz and H.~Diamant, \emph{Soft quasicrystals--why are they stable?}, Philosophical Magazine \textbf{87} (2007), no.~18-21, 3021--3030.

\bibitem{liu2023constrained}
W.~Liu, Z.~Xie, and Y.~Yuan, \emph{A constrained gentlest ascent dynamics and its applications to finding excited states of {Bose--Einstein} condensates}, Journal of Computational Physics \textbf{473} (2023), 111719.

\bibitem{luo2025accelerated}
Y.~Luo, L.~Zhang, and X.~Zheng, \emph{Accelerated high-index saddle dynamics method for searching high-index saddle points}, Journal of Scientific Computing \textbf{102} (2025), no.~2, 31.

\bibitem{luo2022sinum}
Y.~Luo, X.~Zheng, X.~Cheng, and L.~Zhang, \emph{Convergence analysis of discrete high-index saddle dynamics}, SIAM Journal on Numerical Analysis \textbf{60} (2022), no.~5, 2731--2750.

\bibitem{nesterov2013introductory}
Y.~Nesterov, \emph{Introductory lectures on convex optimization: A basic course}, vol.~87, Springer Science \& Business Media, 2013.

\bibitem{nocedal1999numerical}
J.~Nocedal and S.~J. Wright, \emph{Numerical optimization}, Springer, 1999.

\bibitem{onuchic1997theory}
J.~N. Onuchic, Z.~Luthey-Schulten, and P.~G. Wolynes, \emph{Theory of protein folding: the energy landscape perspective}, Annual Review of Physical Chemistry \textbf{48} (1997), no.~1, 545--600.

\bibitem{polyak1964some}
B.~T. Polyak, \emph{Some methods of speeding up the convergence of iteration methods}, Ussr computational mathematics and mathematical physics \textbf{4} (1964), no.~5, 1--17.

\bibitem{provatas2011phase}
N.~Provatas and K.~Elder, \emph{Phase-field methods in materials science and engineering}, John Wiley \& Sons, 2011.

\bibitem{quapp2014locating}
W.~Quapp and J.~M. Bofill, \emph{Locating saddle points of any index on potential energy surfaces by the generalized gentlest ascent dynamics}, Theoretical Chemistry Accounts \textbf{133} (2014), 1--14.

\bibitem{su2025improved}
H.~Su, H.~Wang, L.~Zhang, J.~Zhao, and X.~Zheng, \emph{Improved high-index saddle dynamics for finding saddle points and solution landscape}, SIAM Journal on Numerical Analysis \textbf{63} (2025), no.~4, 1757--1775.

\bibitem{wolynes1995navigating}
P.~G. Wolynes, J.~N. Onuchic, and D.~Thirumalai, \emph{Navigating the folding routes}, Science \textbf{267} (1995), no.~5204, 1619--1620.

\bibitem{yin2020constrained}
J.~Yin, Z.~Huang, and L.~Zhang, \emph{Constrained high-index saddle dynamics for the solution landscape with equality constraints}, Journal of Scientific Computing \textbf{91} (2022), no.~2, 62.

\bibitem{yin2020construction}
J.~Yin, Y.~Wang, J.~Z. Chen, P.~Zhang, and L.~Zhang, \emph{Construction of a pathway map on a complicated energy landscape}, Physical Review Letters \textbf{124} (2020), no.~9, 090601.

\bibitem{yin2021searching}
J.~Yin, B.~Yu, and L.~Zhang, \emph{Searching the solution landscape by generalized high-index saddle dynamics}, Science China Mathematics \textbf{64} (2021), 1801--1816.

\bibitem{2019High}
J.~Yin, L.~Zhang, and P.~Zhang, \emph{High-index optimization-based shrinking dimer method for finding high-index saddle points}, SIAM Journal on Scientific Computing \textbf{41} (2019), no.~6, A3576--A3595.

\bibitem{zhang2012constrained}
J.~Zhang and Q.~Du, \emph{Constrained shrinking dimer dynamics for saddle point search with constraints}, Journal of Computational Physics \textbf{231} (2012), no.~14, 4745--4758.

\bibitem{zhangdu2012}
\bysame, \emph{Shrinking dimer dynamics and its applications to saddle point search}, SIAM Journal on Numerical Analysis \textbf{50} (2012), no.~4, 1899--1921.

\bibitem{zhang2016optimization}
L.~Zhang, Q.~Du, and Z.~Zheng, \emph{Optimization-based shrinking dimer method for finding transition states}, SIAM Journal on Scientific Computing \textbf{38} (2016), no.~1, A528--A544.

\bibitem{zhang2022error}
L.~Zhang, P.~Zhang, and X.~Zheng, \emph{Error estimates for euler discretization of high-index saddle dynamics}, SIAM Journal on Numerical Analysis \textbf{60} (2022), no.~5, 2925--2944.

\bibitem{JJIAM2023}
\bysame, \emph{A model-free shrinking-dimer saddle dynamics for finding saddle point and solution landscape}, Japan Journal of Industrial and Applied Mathematics \textbf{40} (2023), no.~3, 1677--1693.

\end{thebibliography}
\bibliographystyle{amsplain}

\end{document}